\documentclass[11pt]{amsart}

\usepackage{amssymb}
\usepackage{amsmath}
\usepackage{amsthm}
\usepackage{stmaryrd}
\usepackage{amsfonts,mathrsfs}
\usepackage{fullpage}
\usepackage{color}                    
\usepackage{hyperref}                 
\usepackage[all]{xy}
\title{On the ring structure of spark characters}
\author{Ning Hao}


\begin{document}
\maketitle

\newtheorem{thm}{Theorem}[section]
\newtheorem{lem}[thm]{Lemma}
\newtheorem{prop}[thm]{Proposition}
\newtheorem{cor}[thm]{Corollary}
\newtheorem{defn}[thm]{Definition}
\newtheorem{fact}[thm]{Fact}
\newtheorem{claim}[thm]{Claim}
\newtheorem{rem}[thm]{Remark}
\newtheorem{ex}[thm]{Example}
\newtheorem{rmk}[thm]{Remark}

\begin{abstract}
We give a new description of the ring structure on the differential characters of
a smooth manifold via the smooth hyperspark complex. We show the explicit product formula,
and as an application, calculate the product for differential characters of the unit circle.
Applying the presentation of spark classes by smooth hypersparks,
we give an explicit construction of the isomorphism between groups of
spark classes and the $(p,p)$ part of smooth Deligne cohomology groups
associated to a smooth manifold. We then give a new direct proof that
this is an isomorphism of ring structures.
\end{abstract}

{\bf \tableofcontents}

\section{Introduction}
In 2006, Harvey and Lawson \cite{HL1} introduced a homological machine ---
spark complexes and their associated groups of spark classes ---
to study secondary geometric invariants of smooth manifolds.
They introduced a large variety of spark complexes which
appear naturally in geometry, topology and physics and showed their
associated groups of spark classes are all naturally isomorphic. These classes
are collectively called Harvey-Lawson spark characters.
Harvey and Lawson also defined a ring structure on spark characters
via the de Rham-Federer spark complex and established the equivalence of
spark characters with the classical Cheeger-Simons differential characters.
In this paper, we give a new description of the ring structure on spark characters of
a smooth manifold via the smooth hyperspark complex.
As an application, we calculate the product of two degree $0$ characters of the unit circle.

Besides differential characters and spark characters, another ingredient of
secondary geometric invariant theory is Deligne cohomology which was invented by Deligne in 1970s.
In this paper, we focus on ``smooth Deligne cohomology", an analog of Deligne cohomology
defined on any smooth manifold. It is well known there is a natural isomorphism
between differential characters and the $(p,p)$ part of smooth Deligne cohomology i.e.
$H^p_{\mathcal{D}}(X,\mathbb{Z}(p)^{\infty})$.
In this paper, we shall construct an explicit isomorphism between spark characters
and the $(p,p)$ part of smooth Deligne cohomology groups. The equivalency of
smooth Deligne cohomology and differential characters follows as a corollary.
Moreover, using the ring structure we introduce here, we show
that this isomorphism is indeed a ring isomorphism.
We thereby produce a new geometrical definition of the product in smooth Deligne cohomology.

In \cite{HL2}, Harvey and Lawson introduced spark characters of level $p$ on a complex manifold.
These character groups contain analytic Deligne cohomology as subgroups.
We shall study the ring structure on spark characters of level $p$ in a sequel paper \cite{H1}.
Moreover, we shall define a product in Deligne cohomology with an analytic formula induced by
the product in spark characters, and show its equivalency with the original product.
In \cite{H2}, we shall study the Massey products in spark characters which induce Massey products in
Deligne cohomology.

This paper is organized as follows.
In Section 2, we review briefly the concept and basic properties
of homological spark complexes. Three main examples of homological spark complexes
follow in Section 3.
In Section 4, we study the cup product on the total complex of double complex
$\bigoplus_{p,q}C^p(\mathcal{U},\mathcal{E}^q)$ carefully and establish
the ring structure of spark characters via the smooth hyperspark complex. Then
we give explicit examples of the product in Section 5. We calculate the product
of two differential characters of degree $0$ on the unit circle. Using Fourier
expansion, we decompose a general degree $0$ character to a linear combination of
typical characters. The product of two general characters will be presented by
the coefficients of their Fourier expansions. Also, we calculate the product of
two characters of degree $1$ on a 3-dimensional manifold when one of them represents
a flat line bundle.
Applying the \v{C}ech resolution of smooth Deligne complex, we give an explicit
construction of isomorphism between smooth Deligne cohomology and spark characters
in Section 6. Moreover, we check the ring structures on them and show the isomorphism
is a ring isomorphism.

\noindent{\textbf{Acknowledgements.}} I am very grateful to my advisor H. Blaine Lawson
who introduced this subject to me. Also, I would like to thank Zhiwei Yun for helpful conversations.

\section{Homological Spark Complexes}
We introduce the definitions of a homological spark complex and its associated
group of homological spark classes. Note that all cochain complexes in this paper are bounded
cochain complexes of abelian groups.

\begin{defn}
A homological spark complex is a triple of cochain complexes
$(F^*,E^*,I^*)$ together with morphisms given by inclusions
$$I^* \hookrightarrow F^* \hookleftarrow E^*$$
such that
\begin{enumerate}
            \item $I^k\cap E^k={0}$ for $k>0$, $F^k=E^k=I^k=0$ for $k<0$,
            \item $H^*(E^*)\cong H^*(F^*)$.

          \end{enumerate}
\end{defn}
\begin{defn}
In a given spark complex $(F^*,E^*,I^*)$, a homological spark of degree $k$ is
an element $a\in F^k $ which satisfies the spark equation
$$da=e-r$$ where $e\in E^{k+1}$ and $r\in I^{k+1}$.

Two sparks $a$, $a'$ of degree $k$ are equivalent if $$a-a'=db+s$$ for some
$b\in F^{k-1}$ and $s\in I^k$.

The set of equivalence classes is called the group of spark classes
of degree $k$ and denoted by $\hat{\mathbf{H}}^k(F^*,E^*,I^*)$ or $\hat{\mathbf{H}}^k$
for short. Let $[a]$ denote the equivalence class containing the spark $a$.
\end{defn}

\begin{lem}
Each homological spark $a\in F^k$ uniquely determines $e \in E^{k+1}$ and
$r\in I^{k+1}$, and $de=dr=0$.
\end{lem}
\begin{proof}
Uniqueness of $e$ and $r$ is from the fact $I^k\cap E^k={0}$. Taking differential
on the spark equation, we get $de-dr=0$ which means $de=dr=0$.
\end{proof}

We now give the fundamental exact sequences associated to a homological spark complex
$(F^*,E^*,I^*)$. Let $Z_I^k(E^*)$ denote the space of cycles $e\in E^k$ which are $F^*$-homologous
to some $r\in I^k$, i.e. $e-r$ is exact in $F^k$.

Let $\hat{\mathbf{H}}^k_E$ denote the space of spark classes that can be represented by a homological
spark $a\in E^k$.
Let us also define $$H_I^k(F^*)\equiv \text{Image}\{H^k(I^*)\rightarrow H^k(F^*)\}\equiv \text{Ker}
\{H^k(F^*)\rightarrow H^k(F^*/I^*)\},$$
$$H^{k+1}(F^*,I^*)\equiv \text{Ker}\{H^{k+1}(I^*)\rightarrow H^{k+1}(F^*)\}\equiv
\text{Image} \{H^k(F^*/I^*)\rightarrow H^{k+1}(I^*)\}.$$

\begin{prop}\cite{HL1}
There exist well-defined surjective homomorphisms
$$\delta_1: \hat{\mathbf{H}}^k\rightarrow Z_I^{k+1}(E^*) \quad \text{and} \quad
\delta_2: \hat{\mathbf{H}}^k \rightarrow H^{k+1}(I^*)$$ given by
$$\delta_1([a])=e \quad \text {and} \quad \delta_2([a])=[r]$$ where $da=e-r$.

Moreover, associated to any spark complex $(F^*,E^*,I^*)$ is the commutative diagram

\xymatrix{& & & 0 \ar[d] & 0 \ar[d] & 0 \ar[d] &   \\
& & 0 \ar[r] & \frac{H^k(F^*)}{H^k_I(F^*)} \ar[r] \ar[d] &
\hat{\mathbf{H}}^k_E \ar[r] \ar[d] & dE^k \ar[r] \ar[d] & 0 \\
& & 0 \ar[r] & H^k(F^*/I^*) \ar[r] \ar[d] & \hat{\mathbf{H}}^k \ar[r]^{\delta_1}
\ar[d]^{\delta_2} & Z_I^{k+1}(E^*) \ar[r] \ar[d] & 0 \\
& & 0 \ar[r] & H^{k+1}(F^*,I^*) \ar[r]\ar[d] &
H^{k+1}(I^*) \ar[r] \ar[d] & H_I^{k+1}(F^*) \ar[r] \ar[d] & 0 \\
& & & 0 & 0 & 0 &\\}
whose rows and columns are exact.
\end{prop}

\begin{defn}
Two spark complexes $(F^*,E^*,I^*)$ and $(\bar{F}^*,\bar{E}^*,\bar{I}^*)$
are quasi-isomorphic if there exists a commutative diagram of morphisms

\xymatrix{& & & & & I^* \ar @{^{(}->}[r] \ar@{^{(}->}[d]^i &
F^* \ar@{^{(}->}[d] & E^*  \ar @{_{(}->}[l] \ar @{=}[d] \\
& & & & & \bar{I}^* \ar @{^{(}->}[r]  & \bar{F}^*
 & \bar{E}^* \ar @{_{(}->}[l]\\}
inducing an isomorphism $$i^*: H^*(I^*) \stackrel{\cong}\longrightarrow H^*(\bar{I}^*).$$
\end{defn}

\begin{prop}\cite{HL1}
A quasi-isomorphism of spark complexes $(F^*,E^*,I^*)$ and $(\bar{F}^*,\bar{E}^*,\bar{I}^*)$
induces an isomorphism
$$\hat{\mathbf{H}}^k(F^*,E^*,I^*)\cong \hat{\mathbf{H}}^k(\bar{F}^*,\bar{E}^*,\bar{I}^*)$$
of the associated groups of spark classes. Moreover, it induces an isomorphism of the
$3\times 3$ grids associated to the two complexes.
\end{prop}

\section{Spark Characters}
We give our main examples of homological spark complexes and define the Harvey-Lawson
spark characters associated to a smooth manifold.

Let $X$ be a smooth manifold of dimension $n$. Let $\mathcal{E}^k$ denote the sheaf
of smooth differential $k$-forms on $X$, $\mathcal{D}^k$ the sheaf of currents of
degree $k$  on $X$. Let $\mathcal{R}^k$ and $\mathcal{IF}^k$ denote the sheaf of
rectifiable currents of degree $k$ and the sheaf of integrally flat currents of
degree $k$ on $X$ respectively. Note that
$$\mathcal{IF}^k(U)=\{r+ds: r\in \mathcal{R}^k(U) \text{ and } s\in \mathcal{R}^{k-1}(U)\}$$

\subsection{The de Rham-Federer Spark Complex}
\begin{defn}
The de Rham-Federer spark complex associated to a smooth manifold $X$ is obtained by taking
$$F^k=\mathcal{D}'^k(X),\quad E^k=\mathcal{E}^k(X),\quad I^k=\mathcal{IF}^k(X).$$
\end{defn}
\begin{rmk}
The condition $H^k(\mathcal{D}'^*(X))=H^k(\mathcal{E}^*(X))=H^k(X,\mathbb{R})$ is standard.
For a proof that $\mathcal{E}^k(X)\cap\mathcal{IF}^k(X)=\{0\}$ for $k>0$, see \cite{HLZ} Lemma 1.3.
\end{rmk}

\begin{defn}
A de Rham-Federer spark of degree $k$ is a current $a\in\mathcal{D}'^k(X)$ with the spark equation
$$da=e-r$$ where $e\in \mathcal{E}^{k+1}(X)$ is smooth and $r\in\mathcal{IF}^{k+1}(X)$ is integrally flat.

Two sparks $a$ and $a'$ are equivalent if there exist $b\in\mathcal{D}'^{k-1}(X)$ and $s\in \mathcal{IF}^k(X)$
with $$a-a'=db+s.$$

The equivalence class determined by a spark $a$ will be denoted by $[a]$ and the space of spark classes
will be denoted by $\hat{\mathbf{H}}^k_{spark}(X)$.
\end{defn}

Let $\mathcal{Z}_0^k(X)$ denote closed degree $k$ forms on $X$ with integral periods.
Note that $H^k(\mathcal{IF}^*(X))=H^k(X,\mathbb{Z})$. By Proposition 2.4 , we have

\begin{prop}\cite{HLZ}
There exist well-defined surjective homomorphisms
$$\delta_1: \hat{\mathbf{H}}^k(X)\rightarrow \mathcal{Z}_0^{k+1}(X) \quad \text{and} \quad
\delta_2: \hat{\mathbf{H}}^k(X) \rightarrow H^{k+1}(X,\mathbb{Z})$$ given by
$$\delta_1([a])=e \quad \text {and} \quad \delta_2([a])=[r]$$ where $da=e-r$.

Associated to the de Rham-Federer spark complex is the commutative diagram

\xymatrix{& & & 0 \ar[d] & 0 \ar[d] & 0 \ar[d] &   \\
& & 0 \ar[r] & \frac{H^k(X,\mathbb{R})}{H^k_{free}(X,\mathbb{Z})} \ar[r] \ar[d] &
\hat{\mathbf{H}}^k_{\infty}(X) \ar[r] \ar[d] & d\mathcal{E}^k(X) \ar[r] \ar[d] & 0 \\
& & 0 \ar[r] & H^k(X,\mathbb{R}/\mathbb{Z}) \ar[r] \ar[d] & \hat{\mathbf{H}}^k(X) \ar[r]^{\delta_1}
\ar[d]^{\delta_2} & \mathcal{Z}_0^{k+1}(X) \ar[r] \ar[d] & 0 \\
& & 0 \ar[r] & H^{k+1}_{tor}(X,\mathbb{Z}) \ar[r]\ar[d] &
H^{k+1}(X,\mathbb{Z}) \ar[r] \ar[d] & H_{free}^{k+1}(X,\mathbb{Z}) \ar[r] \ar[d] & 0 \\
& & & 0 & 0 & 0 &\\}
\noindent where $\hat{\mathbf{H}}^k_{\infty}(X)\cong\mathcal{E}^k(X)/\mathcal{Z}^k_0(X)$
denote the group of spark classes of degree $k$ which
can be represented by smooth forms.
\end{prop}

\subsection{The Hyperspark Complex and Smooth Hyperspark Complex}
Suppose $\mathcal{U}=\{U_i\}$ is a good cover of $X$ (with each intersection $U_I$ contractible).
We have the \v{C}ech-Current bicomplex $\bigoplus_{p,q\geq0}C^p(\mathcal{U},\mathcal{D}'^q)$.
Now we are concerned with the total complex of \v{C}ech-Current bicomplex
$\bigoplus_{p+q=*}C^p(\mathcal{U},\mathcal{D}'^q)$
with total differential $D=\delta+(-1)^pd$.

\begin{defn}
By the hyperspark complex we mean the spark complex defined as
$$(F^*,E^*,I^*)=
(\bigoplus_{p+q=*}C^p(\mathcal{U},\mathcal{D}'^q),\mathcal{E}^*(X),
\bigoplus_{p+q=*}C^p(\mathcal{U},\mathcal{IF}^q)).$$
\end{defn}

\begin{rmk}
We should verify the triple of complexes above is a spark complex.
There is a natural inclusion
$\mathcal{E}^*(X)\hookrightarrow \bigoplus_{p+q=*}C^p(\mathcal{U},\mathcal{D}'^q)$, given
by $$\mathcal{E}^*(X)\hookrightarrow \mathcal{D}'^*(X) \hookrightarrow C^0(\mathcal{U},\mathcal{D}'^*)
\hookrightarrow \bigoplus_{p+q=*}C^p(\mathcal{U},\mathcal{D}'^q).$$

For $k>0$, $\mathcal{E}^k(X) \cap \bigoplus_{p+q=k}C^p(\mathcal{U},\mathcal{IF}^q))=
\mathcal{E}^k(X) \cap C^0(\mathcal{U},\mathcal{IF}^k))=
\mathcal{E}^k(X) \cap \mathcal{IF}^k(X) =\{0\}$.

And it is easy to see $H^*(F^*)=H^*(\mathcal{D}'^*(X))=H^*(X,\mathbb{R})=H^*(E^*)$,
and also $H^*(I^*)=H^*(C^*(\mathcal{U},\mathbb{Z}))=H^*(X,\mathbb{Z})$.
\end{rmk}

\begin{defn}
A hyperspark of degree $k$ is an element $$a\in\bigoplus_{p+q=k}C^p(\mathcal{U},\mathcal{D}'^q)$$
with the property $$Da=e-r$$ where $e\in\mathcal{E}^{k+1}(X)\subset C^0(\mathcal{U},\mathcal{D}'^{k+1})$
is of bidegree $(0,k+1)$ and $r\in \bigoplus_{p+q=k+1}C^p(\mathcal{U},\mathcal{IF}^q))$.

Two hypersparks $a$ and $a'$ are said to be equivalent if there exists $b\in
\bigoplus_{p+q=k-1}C^p(\mathcal{U},\mathcal{D}'^q)$ and $s\in \bigoplus_{p+q=k}C^p(\mathcal{U},\mathcal{IF}^q))$
satisfying $$a-a'=Db+s.$$

The equivalence class determined by a hyperspark $a$ will be denoted by $[a]$, and the space of hyperspark
classes will be denoted by $\hat{\mathbf{H}}^k_{hyperspark}(X)$.
\end{defn}

\begin{prop}
$$\hat{\mathbf{H}}^k_{spark}(X)\cong \hat{\mathbf{H}}^k_{hyperspark}(X).$$
\end{prop}

\begin{proof}
It is easy to see that there is a natural inclusion from the de Rham-Federer spark complex
to the hyperspark complex which is a quasi-isomorphism.
\end{proof}

We can consider the de Rham-Federer spark complex as a spark subcomplex of the hyperspark complex,
now we introduce another spark subcomplex of the hyperspark complex, which is called
the smooth hyperspark complex.

\begin{defn}
By the smooth hyperspark complex we mean the spark complex
$$(F^*,E^*,I^*)=(\bigoplus_{p+q=*}C^p(\mathcal{U},\mathcal{E}^q),\mathcal{E}^*(X),
C^*(\mathcal{U},\mathbb{Z})).$$

\end{defn}

\begin{defn}
A smooth hyperspark of degree $k$ is an element $$A\in\bigoplus_{p+q=k}C^p(\mathcal{U},\mathcal{E}^q)$$
with the property $$Da=e-r$$ where $e\in\mathcal{E}^{k+1}(X)\subset C^0(\mathcal{U},\mathcal{E}^{k+1})$
is of bidegree $(0,k+1)$ and $r\in C^{k+1}(\mathcal{U},\mathbb{Z})$.

Two smooth hypersparks $a$ and $a'$ are equivalent if there exists $b\in
\bigoplus_{p+q=k-1}C^p(\mathcal{U},\mathcal{E}^q)$ and $s\in C^k(\mathcal{U},\mathbb{Z})$
satisfying $$a-a'=Db+s.$$

The equivalence class determined by a smooth hyperspark $a$ will be denoted by $[a]$, and the space of
smooth hyperspark classes will be denoted by $\hat{\mathbf{H}}^k_{smooth}(X)$.

\end{defn}
One can easily verify that the smooth hyperspark complex is quasi-isomorphic to the hyperspark complex \cite{HL1}.
Hence, we have
\begin{prop}
$$\hat{\mathbf{H}}^k_{smooth}(X)\cong \hat{\mathbf{H}}^k_{hyperspark}(X).$$
\end{prop}

\begin{cor}
$$\hat{\mathbf{H}}^k_{spark}(X)\cong \hat{\mathbf{H}}^k_{smooth}(X).$$
\end{cor}

We can consider the hyperspark complex as a bridge which connects the de Rham-Federer spark complex
and the smooth hyperspark complex.

\subsection{Harvey-Lawson Spark Characters}
We defined three homological spark complexes associated to a smooth manifold $X$, and showed
the natural isomorphisms between the groups of spark classes associated to them. We refer
reader to \cite{HL1} for more very interesting spark complexes whose groups of spark classes
are all isomorphic.
We denote the groups of spark classes by $\hat{\mathbf{H}}^*(X)$ collectively, and call them
the \textbf{Harvey-Lawson spark characters associated to} $X$.

An important fact is that $\hat{\mathbf{H}}^*(X)$ has
a ring structure which is functorial with respect to smooth maps between manifolds.
This ring structure on $\hat{\mathbf{H}}^*(X)$ is defined in \cite{HLZ}
via the de Rham-Federer spark complex. The main technical difficulty is that the
wedge product of two currents may not be well-defined.
However, we can always choose good representatives in the following sense:
\begin{prop}\cite[Proposition 3.1]{HLZ}
Given classes $\alpha\in\hat{\mathbf{H}}^k_{spark}(X)$ and $\beta\in\hat{\mathbf{H}}^l_{spark}(X)$ there exist
representatives $a\in\alpha$ and $b\in\beta$ with $da=e-r$ and $db=f-s$ so that
$a\wedge s$, $r\wedge b$ and $r\wedge s$ are well-defined flat currents on $X$ and $r\wedge s$
is rectifiable.
\end{prop}

\begin{thm}\cite[Theorem 3.5]{HLZ}
Setting $$\alpha*\beta\equiv[a\wedge f+(-1)^{k+1}r\wedge b]=
[a\wedge s+(-1)^{k+1}e\wedge b]\in \hat{\mathbf{H}}^{k+l+1}_{spark}(X) $$
gives $\hat{\mathbf{H}}^*_{spark}(X)$ the structure of a graded commutative ring such that
$\delta_1: \hat{\mathbf{H}}^*_{spark}(X)\rightarrow \mathcal{Z}_0^{*+1}(X)$ and
$\delta_2: \hat{\mathbf{H}}^*_{spark}(X) \rightarrow H^{*+1}(X,\mathbb{Z})$ are ring homomorphisms.
\end{thm}

\begin{proof}
It is easy to verify $$d(a\wedge f+(-1)^{k+1}r\wedge b)=d(a\wedge s+(-1)^{k+1}e\wedge b)=e\wedge f-r\wedge s,$$
$$(a\wedge f+(-1)^{k+1}r\wedge b)-(a\wedge s+(-1)^{k+1}e\wedge b)=(-1)^kd(a\wedge b).$$
So $a\wedge f+(-1)^{k+1}r\wedge b$ and $a\wedge s+(-1)^{k+1}e\wedge b$ are sparks and represent the same
spark class.

To show that the product is independent of choices of representatives, assume the spark
$a'\in \mathcal{D}'^k(X)$ represent the same spark class with $a$ and $da'=e'-r'$.
Then $\exists c\in \mathcal{D}'^{k-1}(X)$ and $t\in \mathcal{IF}^k(X)$ with
$a-a'=dc+t$. We have
$$(a\wedge f+(-1)^{k+1}r\wedge b)-(a'\wedge f+(-1)^{k+1}r'\wedge b)=d(c\wedge f+(-1)^k(t\wedge b))+t\wedge s.$$
By the same calculation we can show the product is also independent of choices of representatives
of the second factor.

We can calculate $$\beta*\alpha=[b\wedge e+(-1)^{l+1}s\wedge a]=
(-1)^{(k+1)(l+1)}[a\wedge s+(-1)^{k+1}e\wedge b]=(-1)^{(k+1)(l+1)}\alpha*\beta,$$
i.e. the product is graded commutative.

Also, it is easy to show the product is associative.
\end{proof}

\begin{thm}\cite{HLZ}
Any smooth map $f:X\rightarrow Y$ between two smooth manifolds induces a
graded ring homomorphism $$f^*:\hat{\mathbf{H}}^*(Y)\rightarrow \hat{\mathbf{H}}^*(X)$$
compatible with $\delta_1$ and $\delta_2$.
Moreover, if $g:Y\rightarrow Z$ is smooth, then $(g\circ f)^*=f^*\circ g^*$.
\end{thm}
So we can consider $\hat{\mathbf{H}}^*(\bullet)$ as a graded ring functor on the category
of smooth manifolds and smooth maps.

\subsection{Cheeger-Simons Differential Characters}
Cheeger and Simons introduced differential characters in their remarkable paper \cite{CS}.

Let $X$ be a smooth manifold. And let $C_k(X)\supset Z_k(X) \supset B_k(X)$ denote the groups of
smooth singular $k$-chains, cycles and boundaries.

\begin{defn}
The group of differential characters of degree $k$ is defined by
$$\hat{H}^k(X)=\{h\in \hom(Z_k(X),\mathbb{R}/\mathbb{Z}):dh\equiv \omega \mod \mathbb{Z}, \text{ for some }
\omega\in  \mathcal{E}^{k+1}(X)\}.$$
\end{defn}
\begin{rmk}
For any $\sigma\in C_{k+1}(X)$, $(dh)(\sigma)=h\circ\partial(\sigma)$. In the definition above,
$dh\equiv \omega \mod \mathbb{Z}$ means
$h\circ\partial(\sigma)\equiv\int_{\Delta_{k+1}}\sigma^*(\omega) \mod \mathbb{Z}$,
$\forall \sigma\in C_{k+1}(X)$.
\end{rmk}

Cheeger and Simons also defined the ring structure on $\hat{H}^k(X)$ and showed the
functoriality of $\hat{H}^k(X)$. Harvey, Lawson and Zweck \cite{HL1}\cite{HLZ}
established the equivalency of differential characters and
spark characters.

\begin{thm}\cite{HL1}\cite{HLZ}
$$\hat{\mathbf{H}}^*(X)\cong\hat{H}^*(X).$$
\end{thm}

\section{Ring Structure via the Smooth Hyperspark Complex}
We introduced the Harvey-Lawson spark characters and
established the ring structure.
In this section, we give a new description of the ring structure via the smooth hyperspark complex.

Consider the smooth hyperspark complex
$$(F^*,E^*,I^*)=(\bigoplus_{p+q=*}C^p(\mathcal{U},\mathcal{E}^q),\mathcal{E}^*(X),
C^*(\mathcal{U},\mathbb{Z})).$$
Recall there is a cup product on the cochain complex $C^*(\mathcal{U},\mathbb{Z})$
which induces the ring structure on $H^*(X,\mathbb{Z})$.

\begin{prop}
For $a\in C^r(\mathcal{U},\mathbb{Z})$ and $b\in C^s(\mathcal{U},\mathbb{Z})$,
we define cup product
$$(a\cup b)_{i_0,...,i_{r+s}}\equiv a_{i_0,...,i_r}\cdot b_{i_r,...,i_{r+s}}.$$
This product induces an associative, graded commutative product on
$\check{H}^*(\mathcal{U},\mathbb{Z})\cong H^*(X,\mathbb{Z})$.
\end{prop}
\begin{proof}
It is easy to verify that $\delta (a\cup b)=\delta a\cup b+(-1)^r a\cup\delta b$
(the Leibniz rule), so
the product descends to cohomology. The associativity is trivial. However,
a direct proof of graded commutativity is quite complicated,
see \cite[Proposition 1.3.7]{Br} and \cite{GH}.
\end{proof}

Now we want to define a cup product on the cochain complex
$(\bigoplus_{p+q=*}C^p(\mathcal{U},\mathcal{E}^q),D=\delta+(-1)^pd)$
which is compatible with products on $\mathcal{E}^*(X)$ and $C^*(\mathcal{U},\mathbb{Z})$
and descends to its cohomology.
A first try is to define
$$(a\cup b)_{i_0,...,i_{r+s}}\equiv a_{i_0,...,i_r}\wedge b_{i_r,...,i_{r+s}}
\text{ for } a\in C^r(\mathcal{U},\mathcal{E}^p) \text{ and }
 b\in C^s(\mathcal{U},\mathcal{E}^q).$$
But it turns out that this cup product does not satisfy the Leibniz rule. We
modify the product and define
$$(a\cup b)_{i_0,...,i_{r+s}}\equiv (-1)^{js}a_{i_0,...,i_r}\wedge b_{i_r,...,i_{r+s}}
\text{ for } a\in C^r(\mathcal{U},\mathcal{E}^j) \text{ and }
 b\in C^s(\mathcal{U},\mathcal{E}^k).$$

\begin{prop}
We define a cup product on the cochain complex
$\bigoplus_{p+q=*}C^p(\mathcal{U},\mathcal{E}^q)$  as
$$(a\cup b)_{i_0,...,i_{r+s}}\equiv (-1)^{js}a_{i_0,...,i_r}\wedge b_{i_r,...,i_{r+s}}
\in C^{r+s}(\mathcal{U},\mathcal{E}^{j+k}),$$
for $a\in C^r(\mathcal{U},\mathcal{E}^j)$ and $b\in C^s(\mathcal{U},\mathcal{E}^k)$.
This product is associative and satisfies the Leibniz rule, hence it induces a product
on its cohomology.
\end{prop}

\begin{proof}
Associativity: for $a\in C^r(\mathcal{U},\mathcal{E}^j)$, $b\in C^s(\mathcal{U},\mathcal{E}^k)$
and $c\in C^t(\mathcal{U},\mathcal{E}^l)$, we have
$$((a\cup b)\cup c)_{i_0,...,i_{r+s+t}}=(-1)^{(j+k)t}(a\cup b)_{i_0,...,i_{r+s}}
\wedge c_{i_{r+s},...,i_{r+s+t}}$$
$$=(-1)^{jt+kt+js}a_{i_0,...,i_r}
\wedge b_{i_r,...,i_{r+s}}\wedge c_{i_{r+s},...,i_{r+s+t}},$$
and
$$(a\cup (b\cup c))_{i_0,...,i_{r+s+t}}=(-1)^{j(s+t}a_{i_0,...,i_r}
\wedge (b\cup c)_{i_{r},...,i_{r+s+t}}$$
$$=(-1)^{js+jt+kt}a_{i_0,...,i_r}
\wedge b_{i_r,...,i_{r+s}}\wedge c_{i_{r+s},...,i_{r+s+t}}.$$
Hence, the associativity follows.

The Leibniz rule: We want to check the Leibniz rule $D(a\cup b)= Da\cup b+(-1)^{r+j}a\cup Db$
for $a\in C^r(\mathcal{U},\mathcal{E}^j)$ and
$b\in C^s(\mathcal{U},\mathcal{E}^k)$. We fix the notation
$$(a\wedge b)_{i_0,...,i_{r+s}}\equiv a_{i_0,...,i_r}\wedge b_{i_r,...,i_{r+s}},$$
i.e. $a\cup b=(-1)^{js}a\wedge b$.

It is easy to check that
$$d(a\wedge b)=da\wedge b+(-1)^ja\wedge db \text{ and }
\delta(a\wedge b)=\delta a\wedge b+(-1)^ra\wedge \delta b.$$

\begin{eqnarray*}
& &D(a\cup b)\\
           &=&(\delta+(-1)^{r+s}d)(a\cup b)\\
           &=&(\delta+(-1)^{r+s}d)(-1)^{js}(a\wedge b)\\
           &=&(-1)^{js}\delta(a\wedge b)+(-1)^{r+s+js}d(a\wedge b)\\
           &=&(-1)^{js}(\delta a\wedge b+(-1)^ra\wedge\delta b)+
           (-1)^{r+s+js}(da\wedge b+(-1)^{j}a\wedge db)\\
           &=&(-1)^{js}\delta a\wedge b+(-1)^{js+r}a\wedge\delta b+
           (-1)^{r+s+js}da\wedge b+(-1)^{r+s+j+js}a\wedge db
               \end{eqnarray*}

\begin{eqnarray*}
& &Da\cup b+(-1)^{r+j}a\cup Db\\
&=&(\delta+(-1)^{r}d)a\cup b+(-1)^{r+j}a\cup (\delta+(-1)^sd)b\\
           &=&\delta a\cup b+(-1)^{r}da\cup b+(-1)^{r+j}a\cup \delta b+(-1)^{r+j+s}a\cup db\\
           &=&(-1)^{js}\delta a\wedge b+(-1)^{r+(j+1)s}da\wedge b+
           (-1)^{r+j+j(s+1)}a\wedge \delta b+(-1)^{r+j+s+js}a\wedge db\\
           &=&(-1)^{js}\delta a\wedge b+(-1)^{r+s+js}da\wedge b+
           (-1)^{js+r}a\wedge\delta b+(-1)^{r+s+j+js}a\wedge db
               \end{eqnarray*}

So the Leibniz rule is verified.
\end{proof}

\begin{rmk}
It is easy to see this product is compatible with products on
$\mathcal{E}^*(X)$ and $C^*(\mathcal{U},\mathbb{Z})$.
\end{rmk}

\begin{thm}
For two smooth hyperspark classes $\alpha\in\hat{\mathbf{H}}^k_{smooth}(X)$
and $\beta\in\hat{\mathbf{H}}^l_{smooth}(X)$, choose representatives
$a\in\alpha$ and $b\in\beta$ with spark equations $Da=e-r$ and $Db=f-s$,
where $$a\in \bigoplus_{p+q=k}C^p(\mathcal{U},\mathcal{E}^q),\quad e\in \mathcal{E}^{k+1}(X)
\subset C^0(\mathcal{U},\mathcal{E}^{k+1}),\quad r\in C^{k+1}(\mathcal{U}, \mathbb{Z})
\subset C^{k+1}(\mathcal{U},\mathcal{E}^0),$$
$$b\in \bigoplus_{p+q=l}C^p(\mathcal{U},\mathcal{E}^q),\quad f\in \mathcal{E}^{l+1}(X)
\subset C^0(\mathcal{U},\mathcal{E}^{l+1}),\quad s\in C^{k+1}(\mathcal{U}, \mathbb{Z})
\subset C^{l+1}(\mathcal{U},\mathcal{E}^{0}).$$
The product $$\alpha*\beta\equiv[a\cup f+(-1)^{k+1}r\cup b]=
[a\cup s+(-1)^{k+1}e\cup b]\in \hat{\mathbf{H}}^{k+l+1}_{spark}(X) $$
is well-defined and
gives $\hat{\mathbf{H}}^*_{smooth}(X)$ the structure of a graded commutative ring such that
$\delta_1: \hat{\mathbf{H}}^*_{smooth}(X)\rightarrow \mathcal{Z}_0^{*+1}(X)$ and
$\delta_2: \hat{\mathbf{H}}^*_{smooth}(X) \rightarrow H^{*+1}(X,\mathbb{Z})$ are ring homomorphisms.
\end{thm}

\begin{proof}
Since the cup product satisfies the Leibniz rule:
$$D(a\cup b)= Da\cup b+(-1)^{\deg a}a\cup Db,$$ we have
\begin{eqnarray*}
& &D(a\cup f+(-1)^{k+1}r\cup b)\\
&=&Da\cup f+(-1)^ka\cup Df+(-1)^{k+1}Dr\cup b+(-1)^{k+1+k+1}r\cup Db\\
           &=&(e-r)\cup f+r\cup (f-s)\\
           &=&e\cup f-r\cup f+r\cup a-r\cup s\\
           &=&e\wedge f-r\cup s.
               \end{eqnarray*}

Similarly, we can check
$$D(a\cup s+(-1)^{k+1}e\cup b)=e\wedge f-r\cup s,$$
and
$$(a\cup f+(-1)^{k+1}r\cup b)-(a\cup s+(-1)^{k+1}e\cup b)=(-1)^kd(a\cup b).$$
Therefore, $a\cup f+(-1)^{k+1}r\cup b$ and
$a\cup s+(-1)^{k+1}e\cup b$ are sparks and represent the same spark class.

Assume the spark $a'$ represent
the same spark class with $a$ and $da'=e'-r'$.
Then $\exists c\in \bigoplus_{p+q=k+1}C^p(\mathcal{U},\mathcal{E}^q)$
and $t\in C^k(\mathcal{U},\mathbb{Z})$ with $a-a'=Dc+t$. We have
$$(a\cup f+(-1)^{k+1}r\cup b)-(a'\cup f+(-1)^{k+1}r'\cup b)=D(c\cup f+(-1)^k(t\cup b))+t\cup s.$$
By the same calculation we can show the product is also independent of choices of representatives
of the second factor.
It is easy to check the associativity. It is not easy to give a direct proof of
graded commutativity. However, we can see the graded commutativity as a corollary of
next theorem.
\end{proof}

\begin{rmk}
Since $f\in C^0(\mathcal{U},\mathcal{E}^{l+1})$ and $r\in C^{k+1}(\mathcal{U},\mathcal{E}^0)$,
we have $$\alpha*\beta\equiv[a\cup f+(-1)^{k+1}r\cup b]=[a\wedge f+(-1)^{k+1}r\wedge b].$$

\end{rmk}

\begin{thm}
The products in Theorem 3.14 and Theorem 4.4 give the same ring structure for
$\hat{\mathbf{H}}^*_{spark}(X)\cong\hat{\mathbf{H}}^*_{smooth}(X)$.
\end{thm}
\begin{proof}
We can define a cup product on the cochain complex
$\bigoplus_{p+q=*}C^p(\mathcal{U},\mathcal{D}'^q)$  as
$$(a\cup b)_{i_0,...,i_{r+s}}\equiv (-1)^{js}a_{i_0,...,i_r}\wedge b_{i_r,...,i_{r+s}},$$
for $a\in C^r(\mathcal{U},\mathcal{D}'^j)$ and $b\in C^s(\mathcal{U},\mathcal{D}'^k)$
whenever all $a_{i_0,...,i_r}\wedge b_{i_r,...,i_{r+s}}$ make sense.

For two spark classes $$\alpha\in\hat{\mathbf{H}}^k_{hyper}(X)\cong
\hat{\mathbf{H}}^k_{spark}(X)\cong\hat{\mathbf{H}}^k_{smooth}(X)$$
and $$\beta\in\hat{\mathbf{H}}^l_{hyper}(X)\cong
\hat{\mathbf{H}}^l_{spark}(X)\cong\hat{\mathbf{H}}^l_{smooth}(X).$$
We can choose representatives of $\alpha$ and $\beta$ by two ways.
1) Choose a smooth hyperspark
$a\in \bigoplus_{p+q=k}C^p(\mathcal{U},\mathcal{E}^q)\subset
\bigoplus_{p+q=k}C^p(\mathcal{U},\mathcal{D}'^q)$ representing $\alpha$;
2) Choose a de Rham-Federer spark
$a'\in \mathcal{D}'^k(X)\subset
\bigoplus_{p+q=l}C^p(\mathcal{U},\mathcal{D}'^q)$ representing $\alpha$.
We choose $b$ and $b'$ correspondingly. Moreover we can choose $a'$ and $b'$ to be "good"
representatives in the sense of Proposition 3.13.

Assume the spark equations for $a$, $b$, $a'$ and $b'$ are
$Da=e-r$, $Db=f-s$, $Da'=e-r'$ and $Db'=f-s'$ where
$e\in \mathcal{E}^{k+1}(X)$, $r\in C^{k+1}(\mathcal{U},\mathbb{Z})$,
$f\in \mathcal{E}^{l+1}(X)$, $s\in C^{l+1}(\mathcal{U},\mathbb{Z})$,
$r'\in \mathcal{IF}^{k+1}(X)$ and  $s'\in \mathcal{IF}^{l+1}(X)$.
Note that all cup products
$a\cup b$, $a\cup f$, $r\cup b$, $a'\cup b'$, $a'\cup f$, $r'\cup b'$, ect. are
well-defined in $\bigoplus_{p+q=*}C^p(\mathcal{U},\mathcal{D}'^q)$.
Then we can define product via hyperspark complex by choosing all representatives in
either the smooth hyperspark complex or the de Rham-Federer complex.

Moreover, the product does not depend on the choices of representatives.
In fact, there exist $c\in \bigoplus_{p+q=k-1}C^p(\mathcal{U},\mathcal{D}'^q)$,
$t\in \bigoplus_{p+q=k}C^p(\mathcal{U},\mathcal{IF}^q)$,
$c'\in \bigoplus_{p+q=l-1}C^p(\mathcal{U},\mathcal{D}'^q)$ and
$t'\in \bigoplus_{p+q=l}C^p(\mathcal{U},\mathcal{IF}^q)$ such that
$a-a'=Dc+t$ and $b-b'=Dc'+t'$, since $a$ and $a'$, $b$ and $b'$ represent
the same spark classes. Then
\begin{eqnarray*}
& &(a\cup f+(-1)^{k+1}r\cup b)-(a'\cup f+(-1)^{k+1}r'\cup b')\\
&=&(a\cup f+(-1)^{k+1}r\cup b)-(a\cup f+(-1)^{k+1}r\cup b')\\
& &\qquad+(a\cup f+(-1)^{k+1}r\cup b')-(a'\cup f+(-1)^{k+1}r'\cup b')\\
           &=&(-1)^{k+1}r\cup (b-b')+(a-a')\cup f+(-1)^{k+1}(r-r')\cup b'\\
           &=&(-1)^{k+1}r\cup (Dc'+t')+(Dc+t)\cup f+(-1)^{k+1}(-Dt)\cup b'\\
           &=&(-1)^{k+1}r\cup t'+D(r\cup c')+D(c\cup f)+t\cup f+(-1)^kD(t\cup b')-t\cup(f-s')\\
           &=&D(r\cup c'+c\cup f+(-1)^kt\cup b')+(-1)^{k+1}r\cup t'+t\cup s'
               \end{eqnarray*}

The calculation above shows $(a\cup f+(-1)^{k+1}r\cup b)$ and
$(a'\cup f+(-1)^{k+1}r'\cup b')$ represent the same spark class
whenever the cup products in the sums $r\cup c'+c\cup f+(-1)^kt\cup b'\in
\bigoplus_{p+q=k+l}C^p(\mathcal{U},\mathcal{D}'^q)$ and
$(-1)^{k+1}r\cup t'+t\cup s'\in
\bigoplus_{p+q=k+l+1}C^p(\mathcal{U},\mathcal{IF}^q)$ are well-defined.
On one hand, it is trivial to see $r\cup c'$, $c\cup f$ and $(-1)^{k+1}r\cup t'$ are
well-defined. On the other hand, because $t$ is only related to $a$ and $a'$,
we always can choose $b'$ with $Db'=f-s'$ such that
$(-1)^kt\cup b'$ and $t\cup s'$ are well-defined.
\end{proof}

\section{Examples}
We shall describe $\hat{\mathbf{H}}^*_{smooth}(X)$  and
calculate the product in some low dimensional cases.

In \cite{HL1}, Harvey and Lawson gave very nice descriptions of $\hat{\mathbf{H}}^*_{smooth}(X)$
and grundles, which we review briefly here.

\textbf{Degree 0:} A smooth hyperspark of degree 0 is an element
$a\in C^0(\mathcal{U},\mathcal{E}^0)$ satisfying the spark equation
$$Da=e-r \text{ with } e\in \mathcal{E}^1(X) \text{ and } r\in C^1(\mathcal{U},\mathbb{Z}).$$
Moreover, $$Da=e-r \quad\Leftrightarrow\quad \delta a=-r\in C^1(\mathcal{U},\mathbb{Z}) \quad\text{and}\quad
da=e\in \mathcal{E}^1(X)\quad\Leftrightarrow\quad \delta a\in C^1(\mathcal{U},\mathbb{Z}).$$

Two smooth hypersparks $a$ and $a'$ are equivalent if and only if
$a-a'\in C^0(\mathcal{U},\mathbb{Z})$. Consider the exponential of a smooth
hyperspark $g\equiv e^{2\pi ia}$. $\delta a\in C^1(\mathcal{U},\mathbb{Z})$ implies
$g$ is a global circle valued function, and
$a-a'\in C^0(\mathcal{U},\mathbb{Z}) \Leftrightarrow e^{2\pi ia}=e^{2\pi ia'}$.
Therefore, we have
$$\hat{\mathbf{H}}^0_{smooth}(X)=\{g:X\rightarrow S^1: g \text{ is smooth }\}.$$

\textbf{Degree 1:} A smooth hyperspark of degree 1 is an element
$$a=a^{0,1}+a^{1,0}\in C^0(\mathcal{U},\mathcal{E}^1)\oplus C^1(\mathcal{U},\mathcal{E}^0)$$
satisfying the spark equation
$$Da=e-r \text{ with } e\in \mathcal{E}^2(X) \text{ and } r\in C^2(\mathcal{U},\mathbb{Z})$$
which is equivalent to equations
$\left\{
  \begin{array}{ll}
    \delta a^{1,0}=-r\in C^2(\mathcal{U},\mathbb{Z})  \\
    \delta a^{0,1}-da^{1,0}=0  \\
    da^{0,1}=e\in \mathcal{E}^2(X)
  \end{array}
\right.$.

If $g=e^{2\pi ia^{1,0}}$ then the spark equation is equivalent to
$\left\{
  \begin{array}{ll}
    \delta g=0  \\
    \delta a^{0,1}-\frac{1}{2\pi i}d\log g=0  \\
    da^{0,1}=e\in \mathcal{E}^2(X)
  \end{array}
\right.$.

Note that we can write $a^{0,1}=\{ a^{0,1}_i\}$ where $a^{0,1}_i\in \mathcal{E}^1(U_i)$,
and $g=\{ g_{ij}\}$ where each $g_{ij}$ is a circle valued function on $U_{ij}$.
Then
$$\delta g=0 \quad\Leftrightarrow\quad g_{jk}g_{ki}g_{ij}=1$$ i.e. $g_{ij}$ are transition functions of
a hermitian line bundle, and
$$\delta a^{0,1}-\frac{1}{2\pi i}d\log g=0 \quad\Leftrightarrow\quad
a^{0,1}_j-a^{0,1}_i=\frac{1}{2\pi i}\frac{dg_{ij}}{g_{ij}}$$ which means $a^{0,1}_i$ is
the connection 1-form on $U_i$.

Therefore, it is easy to see
$\hat{\mathbf{H}}^1_{smooth}(X)=$ the set of hermitian line bundles with hermitian connections.

\textbf{Degree} $\mathbf{n}=\dim X$: From Proposition 3.4, we have
$\hat{\mathbf{H}}^n_{smooth}(X)\cong H^n(X,\mathbb{R}/\mathbb{Z})\cong \mathbb{R}/\mathbb{Z}$.
Furthermore, every spark class of top degree can be represented by a global top form. And integrating
this form over $X$ (modulo $\mathbb{Z}$) gives the isomorphism
$\hat{\mathbf{H}}^n_{smooth}(X) \cong \mathbb{R}/\mathbb{Z}$.

\textbf{Ring structure on} $\hat{\mathbf{H}}^*_{smooth}(S^1)$.

Now we calculate the product $\hat{\mathbf{H}}^0_{smooth}(S^1)\otimes
\hat{\mathbf{H}}^0_{smooth}(S^1)\rightarrow \hat{\mathbf{H}}^1_{smooth}(S^1)$.

Let $X$ be the unit circle $S^1$. Fix a small number $\varepsilon >0$ and
an open cover $\mathcal{U}=\{U_1,U_2,U_3\}$
where $$U_1=\{e^{2\pi it} : t\in (-\varepsilon, \frac13)\},\quad
U_2=\{e^{2\pi it} : t\in (\frac13-\varepsilon, \frac23)\},\quad
U_3=\{e^{2\pi it} : t\in (\frac23-\varepsilon, 1)\}.$$

Let $a=(a_1,a_2,a_3)\in C^0(\mathcal{U},\mathcal{E}^0)=
\mathcal{E}^0(U_1)\oplus\mathcal{E}^0(U_2)\oplus\mathcal{E}^0(U_3)$ be
a smooth hyperspark representing a spark class
$\alpha\in  \hat{\mathbf{H}}^0_{smooth}(S^1)$. Since $\delta a\in C^1(\mathcal{U},\mathbb{Z})$,
we have $$(a_2-a_1)\mid_{U_{12}}\in\mathbb{Z},\quad (a_3-a_1)\mid_{U_{13}}\in\mathbb{Z},\quad
(a_3-a_2)\mid_{U_{23}}\in\mathbb{Z}.$$
Moreover, two smooth hypersparks represent the same spark class if and only if the
difference of them is in $C^0(\mathcal{U},\mathbb{Z})$, so we can choose the representative
$a$ to be of form:
$$a_1=a_2\mid_{U_{12}},\quad a_2=a_3\mid_{U_{23}},\quad a_1+N=a_3\mid_{U_{13}},\quad a_1(x_0)\in[0,1)$$
where $N$ is an integer and $x_0=e^{2\pi i\cdot0}\in U_1$.
It is easy to see the representative of this form is unique for any class.
Assume the spark equation for $a$ is $Da=e-r$ for $e\in \mathcal{E}^1(S^1)$,
$r=(r_{12},r_{23},r_{13})\in C^1(\mathcal{U},\mathbb{Z})$. Then $da=e$ is a global
1-form and $\delta a=(0,0,N)=-r$. If we have another smooth hyperspark $b$ of this form
representing spark class $\beta$ with $db=f$, $\delta b=(0,0,N')$,
then by the product formula in last section,
the product $\alpha\beta$ can be represented by $a\cup f-r\cup b$. In the case $r=0$, i.e. $N=0$,
$a$ is a global function and the product is represented by the global 1-form
$af$. Evaluating the integral $\int_{S^1}af \mod \mathbb{Z}$, we get a number in
$\mathbb{R}/\mathbb{Z}$ which representing the product under the isomorphism
$\hat{\mathbf{H}}^1_{smooth}(S^1)\cong \mathbb{R}/\mathbb{Z}$. To calculate the general
product, we need the following lemma.

Let $\tilde{\mathcal{S}}$ be the set
$\{f\in C^{\infty}(\mathbb{R}): f(x+1)-f(x) \in\mathbb{Z}\}$. In fact,
$\tilde{\mathcal{S}}$ is a group. We say
$f\sim g$ if and only if $f(x)-g(x)\equiv N\in \mathbb{Z}$. Define the quotient group
$\mathcal{S}=\tilde{\mathcal{S}}/\sim$. Note that we can identify
$\mathcal{S}$ with the set
$\{f\in C^{\infty}(\mathbb{R}): 0\leq f(0)<1, f(x+1)-f(x) \in\mathbb{Z}\}$.

\begin{lem}
There exists a group isomorphism $\hat{\mathbf{H}}^0_{smooth}(S^1)\cong\mathcal{S}$.
Moreover, for any $f(x)\in \mathcal{S}$,
we have the decomposition
$$f(x)=Nx+C+\sum_{k=1}^{\infty}(A_k\sin(2\pi kx)+B_k\cos(2\pi kx)).$$
Hence, we have the corresponding decomposition of a spark class.
\end{lem}

\begin{proof}
For any spark class $\alpha\in \hat{\mathbf{H}}^0_{smooth}(S^1)$, there exists
a unique representative $a=(a_1,a_2,a_3)\in C^0(\mathcal{U},\mathcal{E}^0)$ with
$$a_1=a_2\mid_{U_{12}},\quad a_2=a_3\mid_{U_{23}},\quad a_1+N=a_3\mid_{U_{13}},\quad a_1(x_0)\in[0,1).$$
We can lift $a$ to a smooth function $\tilde{a}\in\mathcal{S}$ uniquely, and
establish a $1-1$ correspondence between
$\hat{\mathbf{H}}^0_{smooth}(S^1)$ and the set $\mathcal{S}$.

For any smooth function $f\in \mathcal{S}$ with
$f(x+1)-f(x)=N$ for some $N\in\mathbb{Z}$, $f(x)-Nx$ is periodic.
Hence we have the Fourier expansion
$$f(x)-Nx=C+\sum_{k=1}^{\infty}(A_k\sin(2\pi kx)+B_k\cos(2\pi kx)).$$
On the other hand, under the $1-1$ correspondence, every component of the Fourier
expansion is still in $\mathcal{S}$, and hence represents a spark class.
\end{proof}

Now let us calculate the product
$\hat{\mathbf{H}}^0_{smooth}(S^1)\otimes\hat{\mathbf{H}}^0_{smooth}(S^1)
\rightarrow \hat{\mathbf{H}}^1_{smooth}(S^1)$. We use identification
$\hat{\mathbf{H}}^0_{smooth}(S^1)\cong\mathcal{S}$ and
$\hat{\mathbf{H}}^1_{smooth}(S^1)\cong \mathbb{R}/\mathbb{Z}$, and
represent the product as $\mathcal{S}\otimes \mathcal{S}\rightarrow \mathbb{R}/\mathbb{Z}$.

\begin{thm}
For $a,b\in \mathcal{S}$ with decompositions
$$a=Nx+C+\sum_{k=1}^{\infty}(A_k\sin(2\pi kx)+B_k\cos(2\pi kx))$$ and
$$b=N'x+C'+\sum_{k=1}^{\infty}(A_k'\sin(2\pi kx)+B_k'\cos(2\pi kx)),$$
the product $$a*b=\frac{NN'}{2}+CN'-C'N+\sum_{k=1}^{\infty}(A_k'B_k-A_kB_k')\pi k \mod \mathbb{Z}.$$
\end{thm}
\begin{proof}
First, we calculate the product $[\sin2\pi kx]*[\cos2\pi k'x]$. Since $\sin2\pi kx$
corresponds to a smooth hyperspark
$a_k=(\sin2\pi kx\mid_{U_1},\sin2\pi kx\mid_{U_2},\sin2\pi kx\mid_{U_3})$ with
spark equations $$da_k=d\sin2\pi kx = 2\pi k\cos2\pi kxdx\text{ and } \delta a_k=(0,0,0)=0.$$
Similarly, $\cos2\pi k'x$ corresponds a smooth hyperspark $b_{k'}$ with
$$db_{k'}=-2\pi k'\sin2\pi k'xdx\text{ and } \delta b_{k'}= 0.$$
Then by the product formula we have
\begin{eqnarray*}
& &[\sin2\pi kx]*[\cos2\pi k'x]\\
&=&\int_0^1\sin2\pi kxd\cos2\pi k'x\\
&=&-2\pi k'\int_0^1\sin2\pi kx \sin2\pi k'xdx\\
&=&-2\pi k'\int_0^1\frac12(\cos2\pi(k-k')x-\cos2\pi(k+k')x)dx\\
&=&\left\{
     \begin{array}{ll}
       -\pi k, & k=k' \\
       0, & \hbox{otherwise.}
     \end{array}
   \right.
               \end{eqnarray*}

Similarly, we can calculate
\begin{eqnarray*}
& &[C]*[Nx]=\int_0^1CNdx=CN\\
& &[C]*[\sin2\pi kx]=\int_0^1Cd\sin2\pi kx=0\\
& &[C]*[\cos2\pi kx]=\int_0^1Cd\cos2\pi kx=0\\
& &[C]*[C']=0\\
& &[\sin2\pi kx]*[\sin2\pi k'x]=\int_0^1\sin2\pi kxd\sin2\pi k'x=0\\
& &[\cos2\pi kx]*[\cos2\pi k'x]=\int_0^1\cos2\pi kxd\cos2\pi k'x=0\\
& &[\sin2\pi kx]*[Nx]=\int_0^1\sin2\pi kxNdx=0\\
& &[\cos2\pi kx]*[Nx]=\int_0^1\cos2\pi kxNdx=0
               \end{eqnarray*}

$Nx$ corresponds to a smooth hyperspark
$z_n=(Nx\mid_{U_1},Nx\mid_{U_2},Nx\mid_{U_3})$ with
spark equations $$dz_n=dNx = Ndx\text{ and } \delta z_n=(0,0,N).$$
So the product of $Nx$ and $N'x$ can be represented by smooth hyperspark
$$NxdN'x+(-1)^1(-(0,0,N))N'x=NN'xdx+(0,0,NN'x)\in C^0(\mathcal{U},\mathcal{E}^1)\oplus
C^1(\mathcal{U},\mathcal{E}^0).$$

Let $\frac12NN'x^2$ denote the element
$$(\frac12NN'x^2\mid_{U_1},\frac12NN'x^2\mid_{U_2},\frac12NN'x^2\mid_{U_3})
\in \mathcal{E}^0(U_1)\oplus\mathcal{E}^0(U_2)\oplus\mathcal{E}^0(U_3)=C^0(\mathcal{U},\mathcal{E}^0).$$
Then $$D(\frac12NN'x^2)=d(\frac12NN'x^2)+\delta(\frac12NN'x^2)=NN'xdx+(0,0,NN'(x+\frac12)).$$
Hence, $NN'xdx+(0,0,NN'x)$ is equivalent to
$$NN'xdx+(0,0,NN'x)-D(\frac12NN'x^2)=0-(0,0,\frac12NN').$$
And $-(0,0,\frac12NN')\in C^1(\mathcal{U},\mathbb{R})$ equals
$-\frac12NN'\equiv\frac12NN' \mod \mathbb{Z}$ under the isomorphism
$$H^1(S^1,\mathbb{R})/H^1(S^1,\mathbb{Z})\cong H^1(S^1,\mathbb{R}/\mathbb{Z})\cong \mathbb{R}/\mathbb{Z}.$$

So we have
$$[Nx]*[N'x]=\frac{NN'}{2}$$

Finally, by distributivity of graded commutativity of the product, we have

$$[Nx+C+\sum_{k=1}^{\infty}(A_k\sin(2\pi kx)+B_k\cos(2\pi kx))]
[N'x+C'+\sum_{k=1}^{\infty}(A_k'\sin(2\pi kx)+B_k'\cos(2\pi kx))]$$
$$=\frac{NN'}{2}+CN'-C'N+\sum_{k=1}^{\infty}(A_k'B_k-A_kB_k')\pi k \mod \mathbb{Z}$$

\end{proof}

The next example we shall discuss is the product of two smooth hypersparks of degree 1
on a 3-dimensional manifold $X$. Since $\hat{\mathbf{H}}^1_{smooth}(X)$ is
the set of hermitian line bundles with hermitian connections and
$\hat{\mathbf{H}}^3_{smooth}(X)\cong \mathbb{R}/\mathbb{Z}$, the product
associates a number modulo $\mathbb{Z}$ to two
hermitian line bundles with hermitian connections.

For two smooth hyperspark classes $\alpha,\beta \in \hat{\mathbf{H}}^1_{smooth}(X)$,
assume
$$a=a^{0,1}+a^{1,0}\in C^0(\mathcal{U},\mathcal{E}^1)\oplus C^1(\mathcal{U},\mathcal{E}^0)
\text{ and }
b=b^{0,1}+b^{1,0}\in C^0(\mathcal{U},\mathcal{E}^1)\oplus C^1(\mathcal{U},\mathcal{E}^0)$$
are representatives of $\alpha$ and $\beta$ respectively with spark equations
$$Da=e-r \text{ and } Db=f-s.$$ Then we have
$\left\{
  \begin{array}{ll}
    \delta a^{1,0}=-r\in C^2(\mathcal{U},\mathbb{Z})  \\
    \delta a^{0,1}-da^{1,0}=0  \\
    da^{0,1}=e\in \mathcal{E}^2(X)
  \end{array}
\right.$
and
$\left\{
  \begin{array}{ll}
    \delta b^{1,0}=-s\in C^2(\mathcal{U},\mathbb{Z})  \\
    \delta b^{0,1}-db^{1,0}=0  \\
    db^{0,1}=f\in \mathcal{E}^2(X)
  \end{array}
\right.$.

By product formula, we have $\alpha\beta=[a\cup f+r\cup b]$ where
$$a\cup f+r\cup b=a^{0,1}\wedge f+a^{1,0}\wedge f+r\wedge b^{0,1}+r\wedge b^{1,0}\in
C^0(\mathcal{U},\mathcal{E}^3)\oplus C^1(\mathcal{U},\mathcal{E}^2)\oplus
C^2(\mathcal{U},\mathcal{E}^1)\oplus C^3(\mathcal{U},\mathcal{E}^0).$$
$a\cup f+r\cup b$ is a cycle in $\bigoplus_{i+j=3}C^i(\mathcal{U},\mathcal{E}^j)$ representing
a class in $H^3(X,\mathbb{R})\cong \mathbb{R}$. In general, it is hard to calculate
the class. However, when one of $\alpha$ and $\beta$ represents a flat bundle,
we can calculate their product.

\begin{lem}
If $\beta\in H^1(X,\mathbb{R}/\mathbb{Z})\subset \hat{\mathbf{H}}^1_{smooth}(X)$ represents
a flat bundle on $X$, then their exists a smooth hyperspark $b=b^{0,1}+b^{1,0}$ representing
$\beta$ with $b^{0,1}=0$ and $b^{1,0}\in  C^1(\mathcal{U},\mathbb{R})$.
\end{lem}
\begin{proof}
For any flat line bundle, there exists a trivialization with constant transition functions and
zero connection forms (with respect to local basis).
\end{proof}

By the Lemma, if $\beta$ is flat, we have $\alpha\beta=[a\cup f+r\cup b]=[r\wedge b^{1,0}]$
where $r\wedge b^{1,0}\in C^3(\mathcal{U},\mathbb{R})\subset C^3(\mathcal{U},\mathcal{E}^0)$
is a \v{C}ech cycle representing a cohomology class in $H^3(X,\mathbb{R})$. Hence,
we proved the following proposition.

\begin{prop}
$X$ is a 3-dimensional manifold. Let $\alpha \in \hat{\mathbf{H}}^1_{smooth}(X)$ and
$\beta\in H^1(X,\mathbb{R}/\mathbb{Z})\subset \hat{\mathbf{H}}^1_{smooth}(X)$. Choosing
representatives as above, we have
$\alpha\beta=[r\wedge b^{1,0}]\in H^3(X,\mathbb{R}/\mathbb{Z})\cong\mathbb{R}/\mathbb{Z}$.
\end{prop}

\begin{rmk}
It is easy to generalize this Proposition to the product
$$\hat{\mathbf{H}}^{n-2}_{smooth}(X)\otimes \hat{\mathbf{H}}^1_{smooth}(X)
\rightarrow \hat{\mathbf{H}}^n_{smooth}(X)$$ for an $n$-dimensional manifold $X$
when the second factor $\beta\in  H^1(X,\mathbb{R}/\mathbb{Z})\subset \hat{\mathbf{H}}^1_{smooth}(X)$.
\end{rmk}

\begin{rmk}
From this Proposition, we see the product $\alpha\beta$ only depends the first chern class $[r]$
of $\alpha$. We can also see this fact from the next Lemma.

Moreover, the product coincides with the natural product
$H^2(X,\mathbb{Z})\otimes H^1(X,\mathbb{R}/\mathbb{Z})\rightarrow
H^3(X,\mathbb{R}/\mathbb{Z})$.
\end{rmk}

\begin{lem}
$X$ is a smooth manifold. If $\alpha\in  \hat{\mathbf{H}}^k_{\infty}(X)
\subset \hat{\mathbf{H}}^k_{smooth}(X)$ and
$\beta\in H^l(X,\mathbb{R}/\mathbb{Z})\subset \hat{\mathbf{H}}^l_{smooth}(X)$,
then $\alpha\beta=0$.
\end{lem}
\begin{proof}
Note that $H^l(X,\mathbb{R}/\mathbb{Z})=\ker\delta_1$ and
$\hat{\mathbf{H}}^k_{\infty}(X)=\ker\delta_2$. So we can choose representatives
$a$ and $b$ with spark equations $Da=e-0$ and $Db=0-s$. By the product formula
we have $\alpha\beta=0$.
\end{proof}

\section{Smooth Deligne Cohomology}
Deligne cohomology, which was invented by Deligne in 1970's, is closely related to
spark characters and differential characters. In this section,
we introduce ``smooth Deligne cohomology"\cite{Br}, a smooth analog of
Deligne cohomology and establish its relation with spark characters.

\begin{defn}
Let $X$ be a smooth manifold. For $p\geq 0$, the smooth Deligne complex
$\mathbb{Z}_{\mathcal{D}}(p)^{\infty}$ is the complex of sheaves:
$$0\rightarrow\mathbb{Z}\stackrel{i}\rightarrow\mathcal{E}^0\stackrel{d}\rightarrow\mathcal{E}^1
\stackrel{d}\rightarrow\cdots\stackrel{d}\rightarrow\mathcal{E}^{p-1}\rightarrow 0$$
where $\mathcal{E}^k$ denotes the sheaf of real-valued differential $k$-forms on $X$.
The hypercohomology groups $\mathbb{H}^q(X,\mathbb{Z}_{\mathcal{D}}(p)^{\infty})$ are called
the smooth Deligne cohomology groups of $X$, and are denoted by $H^q_{\mathcal{D}}(X,\mathbb{Z}(p)^{\infty})$.
\end{defn}

\begin{ex}
It is easy to see $H^q_{\mathcal{D}}(X,\mathbb{Z}(0)^{\infty})=H^q(X,\mathbb{Z})$ and
$H^q_{\mathcal{D}}(X,\mathbb{Z}(1)^{\infty})=H^{q-1}(X,\mathbb{R}/\mathbb{Z})$.
\end{ex}

There is a cup product \cite{Br} \cite{EV}
$$\cup: \mathbb{Z}_{\mathcal{D}}(p)^{\infty}\otimes \mathbb{Z}_{\mathcal{D}}(p')^{\infty}
\rightarrow \mathbb{Z}_{\mathcal{D}}(p+p')^{\infty}$$ by
$$x\cup y=\left\{
    \begin{array}{ll}
      x\cdot y & \hbox{if } \deg x=0;\\
      x\wedge dy & \hbox{if } \deg x>0 \hbox{ and } \deg y=p';\\
      0 & \hbox{otherwise.}
    \end{array}
  \right.$$
$\cup$ is a morphism of complexes and associative, hence induces a ring structure on
$$\bigoplus_{p,q}H^q_{\mathcal{D}}(X,\mathbb{Z}(p)^{\infty}).$$

We may calculate the smooth Deligne cohomology groups of a manifold $X$ with dimension $n$ by the following
two short exact sequences of complexes of sheaves:
\begin{enumerate}
  \item $0\rightarrow\mathcal{E}^{*<p}[-1]\rightarrow\mathbb{Z}_{\mathcal{D}}(p)^{\infty}
  \rightarrow\mathbb{Z}\rightarrow 0$,
  \item $0\rightarrow\mathcal{E}^{*\geq p}[-p-1]\rightarrow\mathbb{Z}_{\mathcal{D}}(n+1)^{\infty}
  \rightarrow\mathbb{Z}_{\mathcal{D}}(p)^{\infty}\rightarrow0$
\end{enumerate}
where $\mathcal{E}^{*<p}[-1]$ denotes the complex of sheaves
$\mathcal{E}^0\stackrel{d}\rightarrow\mathcal{E}^1\stackrel{d}\rightarrow\cdots
\stackrel{d}\rightarrow\mathcal{E}^{p-1}$ shifted by $1$ position to the right,
and $\mathcal{E}^{*\geq p}[-p-1]$ denotes the complex of sheaves
$\mathcal{E}^p\stackrel{d}\rightarrow\mathcal{E}^{p+1}\stackrel{d}\rightarrow
\cdots\stackrel{d}\rightarrow\mathcal{E}^{n}$ shifted by $p+1$ positions to the right.

It turns out $H^p_{\mathcal{D}}(X,\mathbb{Z}(p)^{\infty})$ is the most
interesting part among all the smooth Deligne cohomology groups.

\begin{thm}
We can put $H^p_{\mathcal{D}}(X,\mathbb{Z}(p)^{\infty})$ into the following
two short exact sequences:
\begin{enumerate}
  \item $0\longrightarrow \mathcal{E}^{p-1}(X)/\mathcal{Z}_0^{p-1}(X)\longrightarrow
      H^p_{\mathcal{D}}(X,\mathbb{Z}(p)^{\infty}) \longrightarrow H^{p}(X,\mathbb{Z})\longrightarrow 0$
  \item $0\longrightarrow H^{p-1}(X,\mathbb{R}/\mathbb{Z})\longrightarrow
      H^p_{\mathcal{D}}(X,\mathbb{Z}(p)^{\infty}) \longrightarrow \mathcal{Z}_0^{p}(X)\longrightarrow 0$
\end{enumerate}
\end{thm}

\begin{proof}
(1) From the short exact sequence  $0\rightarrow\mathcal{E}^{*<p}[-1]\rightarrow\mathbb{Z}_{\mathcal{D}}(p)^{\infty}
  \rightarrow\mathbb{Z}\rightarrow 0$, we get the long exact sequence of hypercohomology:
  $$\cdots\rightarrow \mathbb{H}^{p-1}(\mathbb{Z}) \rightarrow \mathbb{H}^p(\mathcal{E}^{*<p}[-1])
  \rightarrow \mathbb{H}^p(\mathbb{Z}_{\mathcal{D}}(p)^{\infty}) \rightarrow
  \mathbb{H}^p(\mathbb{Z}) \rightarrow \mathbb{H}^{p+1} (\mathcal{E}^{*<p}[-1])\rightarrow\cdots.$$

  First, we have $\mathbb{H}^p(\mathbb{Z})=H^p(X,\mathbb{Z})$. Since the sheaf $\mathcal{E}^k$ is soft
  for every $k$, it is easy to see
  $\mathbb{H}^p(\mathcal{E}^{*<p}[-1])=\mathbb{H}^{p-1}(\mathcal{E}^{*<p})=
  \mathcal{E}^{p-1}(X)/d\mathcal{E}^{p-2}(X)$ and
  $\mathbb{H}^{p+1}(\mathcal{E}^{*<p}[-1])=\mathbb{H}^{p}(\mathcal{E}^{*<p})=0$. And
  $\mathbb{H}^p(\mathbb{Z}_{\mathcal{D}}(p)^{\infty})=H^p_{\mathcal{D}}(X,\mathbb{Z}(p)^{\infty})$
  by notation. So we have
  $$\cdots\rightarrow H^{p-1}(X,\mathbb{Z}) \rightarrow \mathcal{E}^{p-1}(X)/d\mathcal{E}^{p-2}(X)
  \rightarrow H^p_{\mathcal{D}}(X,\mathbb{Z}(p)^{\infty}) \rightarrow
  H^p(X,\mathbb{Z}) \rightarrow 0.$$
  Note that the map $H^{p-1}(X,\mathbb{Z}) \rightarrow \mathcal{E}^{p-1}(X)/d\mathcal{E}^{p-2}(X)$
  is induced by morphism of complexes of sheaves $i:\mathbb{Z}\rightarrow \mathcal{E}^{*<p}$
  which is composition of $i:\mathbb{Z}\rightarrow \mathcal{E}^{*}$ and projection
  $p:\mathcal{E}^{*}\rightarrow \mathcal{E}^{*<p}$. Hence
  $H^{p-1}(X,\mathbb{Z}) \rightarrow \mathcal{E}^{p-1}(X)/d\mathcal{E}^{p-2}(X)$ factors through
  $\mathbb{H}^{p-1}(\mathcal{E}^*)=H^{p-1}(X,\mathbb{R})$, and the image is
  $\mathcal{Z}_0^{p-1}(X)/d\mathcal{E}^{p-2}(X)$. Finally, we get the short exact sequence
  $$0\longrightarrow \mathcal{E}^{p-1}(X)/\mathcal{Z}_0^{p-1}(X)\longrightarrow
  H^p_{\mathcal{D}}(X,\mathbb{Z}(p)^{\infty}) \longrightarrow H^{p}(X,\mathbb{Z})\longrightarrow 0.$$

(2) From the short exact sequence  $0\rightarrow\mathcal{E}^{*\geq p}[-p-1]\rightarrow\mathbb{Z}_ {\mathcal{D}}(n+1)^{\infty} \rightarrow\mathbb{Z}_{\mathcal{D}}(p)^{\infty}\rightarrow0$,
 we get the long exact sequence of hypercohomology:
  $$\cdots\rightarrow \mathbb{H}^{p}(\mathcal{E}^{*\geq p}[-p-1]) \rightarrow
  \mathbb{H}^p(\mathbb{Z}_ {\mathcal{D}}(n+1)^{\infty})
  \rightarrow \mathbb{H}^p(\mathbb{Z}_{\mathcal{D}}(p)^{\infty}) \rightarrow $$
  $$\mathbb{H}^{p+1}(\mathcal{E}^{*\geq p}[-p-1]) \rightarrow
  \mathbb{H}^{p+1} (\mathbb{Z}_ {\mathcal{D}}(n+1)^{\infty})\rightarrow\cdots.$$

The complex of sheaves $\mathbb{Z}_ {\mathcal{D}}(n+1)^{\infty}$ is quasi-isomorphic to
$\mathbb{R}/\mathbb{Z}[-1]$, so $\mathbb{H}^p(\mathbb{Z}_ {\mathcal{D}}(n+1)^{\infty})=
H^{p-1}(X,\mathbb{R}/\mathbb{Z})$. Also, it is easy to see
$\mathbb{H}^{p+1}(\mathcal{E}^{*\geq p}[-p-1])=\mathbb{H}^{0}(\mathcal{E}^{*\geq p})=
\mathcal{Z}^p(X)$, and $\mathbb{H}^{p}(\mathcal{E}^{*\geq p}[-p-1])=0$. Thus, we get
$$0\rightarrow H^{p-1}(X,\mathbb{R}/\mathbb{Z})\rightarrow H^p_{\mathcal{D}}(X,\mathbb{Z}(p)^{\infty})
\rightarrow \mathcal{Z}^p(X)\rightarrow H^{p}(X,\mathbb{R}/\mathbb{Z})\rightarrow \cdots.$$

To complete our proof, we have to determine the kernel of the map
$\mathcal{Z}^p(X)\rightarrow H^{p}(X,\mathbb{R}/\mathbb{Z})$. Note this map is induced by
$i: \mathcal{E}^{*\geq p}[-p-1]\rightarrow\mathbb{Z}_ {\mathcal{D}}(n+1)^{\infty}$, which
is composition of $i: \mathcal{E}^{*\geq p}[-p-1]\rightarrow\mathcal{E}^{*}[-1]$
and $i: \mathcal{E}^{*}[-1]\rightarrow\mathbb{Z}_ {\mathcal{D}}(n+1)^{\infty}$. Thereafter,
the map $\mathcal{Z}^p(X)\rightarrow H^{p}(X,\mathbb{R}/\mathbb{Z})$ is composition of
$\mathcal{Z}^p(X)\rightarrow H^{p}(X,\mathbb{R})$ and
$H^{p}(X,\mathbb{R})\rightarrow H^{p}(X,\mathbb{R}/\mathbb{Z})$, and it is easy to
see the kernel is $\mathcal{Z}_0^{p}(X)$.

\end{proof}

\begin{rmk}
By similar calculations, it is easy to determine other part of the smooth Deligne cohomology groups:
$$H^q_{\mathcal{D}}(X,\mathbb{Z}(p)^{\infty})= \left\{
  \begin{array}{ll}
    H^{q-1}(X,\mathbb{R}/\mathbb{Z}), & \hbox{when } (q < p); \\
    H^q(X,\mathbb{Z}), & \hbox{when } (q > p).
  \end{array}
\right.$$

\end{rmk}

We see the $(p,p)$-part of smooth Deligne cohomology satisfy the same
short exact sequences with spark characters (Proposition 3.4).
It is not surprising we have the isomorphism:

\begin{thm}
$$H^p_{\mathcal{D}}(X,\mathbb{Z}(p)^{\infty})\cong\hat{\mathbf{H}}^{p-1}(X).$$
\end{thm}
\begin{proof}
It suffices to show the isomorphism
$H^p_{\mathcal{D}}(X,\mathbb{Z}(p)^{\infty})\cong\hat{\mathbf{H}}^{p-1}_{smooth}(X)$.

Step 1: Choose a good cover $\{\mathcal{U}\}$ of $X$ and
take \v{C}ech resolution for the complex of sheaves
$\mathbb{Z}_{\mathcal{D}}(p)^{\infty}\longrightarrow
\mathcal{C}^*(\mathcal{U},\mathbb{Z}_{\mathcal{D}}(p)^{\infty})$.

Then
$$H^q_{\mathcal{D}}(X,\mathbb{Z}(p)^{\infty})\equiv \mathbb{H}^q(\mathbb{Z}_{\mathcal{D}}(p)^{\infty})
\cong \mathbb{H}^q(Tot(\mathcal{C}^*(\mathcal{U},\mathbb{Z}_{\mathcal{D}}(p)^{\infty})))
\cong H^q(Tot(C^*(\mathcal{U},\mathbb{Z}_{\mathcal{D}}(p)^{\infty})))$$
where $C^*(\mathcal{U},\mathbb{Z}_{\mathcal{D}}(p)^{\infty})$ are the groups of global sections of
sheaves $\mathcal{C}^*(\mathcal{U},\mathbb{Z}_{\mathcal{D}}(p)^{\infty})$ and look like the
following double complex.

\xymatrix{
\vdots & \vdots  & \vdots  & \vdots  &  & \vdots    \\
C^p(\mathcal{U},\mathbb{Z}) \ar[r]^{(-1)^{p}i} \ar[u]^{\delta}
& C^p(\mathcal{U},\mathcal{E}^0) \ar[r]^{(-1)^{p}d} \ar[u]^{\delta}
& C^p(\mathcal{U},\mathcal{E}^1) \ar[r]^{(-1)^{p}d} \ar[u]^{\delta} &
C^p(\mathcal{U},\mathcal{E}^2)  \ar[r]^{(-1)^{p}d} \ar[u]^{\delta} & \cdots \ar[r]^{(-1)^{p}d} &
C^p(\mathcal{U},\mathcal{E}^{p-1})  \ar[u]^{\delta}   \\
\vdots  \ar[u]^{\delta}
& \vdots  \ar[u]^{\delta}
& \vdots \ar[u]^{\delta} &
\vdots \ar[u]^{\delta} &  &
\vdots \ar[u]^{\delta}   \\
C^2(\mathcal{U},\mathbb{Z}) \ar[r]^i \ar[u]^{\delta}
& C^2(\mathcal{U},\mathcal{E}^0) \ar[r]^{d} \ar[u]^{\delta}
& C^2(\mathcal{U},\mathcal{E}^1) \ar[r]^{d} \ar[u]^{\delta} &
C^2(\mathcal{U},\mathcal{E}^2)  \ar[r]^d \ar[u]^{\delta} & \cdots \ar[r]^d &
C^2(\mathcal{U},\mathcal{E}^{p-1})  \ar[u]^{\delta}   \\
C^1(\mathcal{U},\mathbb{Z}) \ar[r]^{-i} \ar[u]^{\delta} & C^1(\mathcal{U},\mathcal{E}^0) \ar[r]^{-d}
\ar[u]^{\delta} & C^1(\mathcal{U},\mathcal{E}^1) \ar[r]^{-d} \ar[u]^{\delta} &
C^1(\mathcal{U},\mathcal{E}^2) \ar[r]^{-d} \ar[u]^{\delta} & \cdots \ar[r]^{-d} &
C^1(\mathcal{U},\mathcal{E}^{p-1})  \ar[u]^{\delta}   \\
C^0(\mathcal{U},\mathbb{Z}) \ar[r]^i \ar[u]^{\delta} & C^0(\mathcal{U},\mathcal{E}^0) \ar[r]^{d}
\ar[u]^{\delta} & C^0(\mathcal{U},\mathcal{E}^1) \ar[r]^{d} \ar[u]^{\delta} &
C^0(\mathcal{U},\mathcal{E}^2) \ar[r]^d \ar[u]^{\delta} & \cdots \ar[r]^d &
C^0(\mathcal{U},\mathcal{E}^{p-1})  \ar[u]^{\delta} \\}

\vspace{0.5cm}
Step 2:

Let $M^*_p \equiv Tot(C^*(\mathcal{U},\mathbb{Z}_{\mathcal{D}}(p)^{\infty}))$
denote the total complexes of the double complex
$C^*(\mathcal{U},\mathbb{Z}_{\mathcal{D}}(p)^{\infty})$ with differential
$$D_p(a)=\left\{
                     \begin{array}{ll}
                       (\delta+(-1)^ri)(a), & \hbox{when}\quad a\in C^r(\mathcal{U},\mathbb{Z});\\
                       (\delta+(-1)^rd)(a), & \hbox{when}\quad a\in C^r(\mathcal{U},\mathcal{E}^j),j<p-1;\\
                       \delta a, & \hbox{when}\quad a\in C^r(\mathcal{U},\mathcal{E}^{p-1}).
                     \end{array}
                   \right.
$$

Now, we show $H^p(M^*_p)\cong \hat{\mathbf{H}}^{p-1}_{smooth}(X)$.

Let $\tilde{a}=r+a=r+\Sigma_{i=0}^{p-1}a^{i,p-1-i}\in M^p_p$ where
$r\in C^p(\mathcal{U},\mathbb{Z})$ and $a^{i,p-1-i}\in C^i(\mathcal{U},\mathcal{E}^{p-1-i})$.
We define a map $H^p(M^*_p)\longrightarrow \hat{\mathbf{H}}^{p-1}_{smooth}(X)$ which
maps $[\tilde{a}]\mapsto[a]$ for $\tilde{a}\in\ker D_p$.

$$\tilde{a}\in \ker D_p \Leftrightarrow D_p \tilde{a}=0
\Leftrightarrow D_pa+(-1)^p i(r)=0 \text{ and } \delta r=0$$

$$\Leftrightarrow Da=D_pa+da^{0,p-1}=da^{0,p-1}-(-1)^p r.$$
Note that $$\delta a^{0,p-1}-da^{1,p-2}=0\quad\Rightarrow \quad
\delta da^{0,p-1}=d\delta a^{0,p-1}=dda^{1,p-2}=0$$

$$\Rightarrow \quad da^{0,p-1}\in \mathcal{E}^p(X)=\ker\delta:
C^0(\mathcal{U},\mathcal{E}^{p})\rightarrow C^1(\mathcal{U},\mathcal{E}^{p}).$$

Therefore, $\tilde{a}\in \ker D_p$ implies $a$ is a smooth hyperspark of degree $p-1$. On the other hand,
if $a$ is a smooth hyperspark with spark equation $Da=e-r$ with
$e\in \mathcal{E}^p(X)$ and $r\in C^p(\mathcal{U},\mathbb{Z})$, it is clear to see
$\tilde{a}\equiv(-1)^pr+a \in \ker D_p$.

Moreover, it is easy to see $\tilde{a}'=r'+a'\in \ker D_p$ with
$\tilde{a}-\tilde{a}'\in Im D_p$ if and only if $a$ and $a'$ represent
the same spark class.

Hence, the map $[\tilde{a}]\rightarrow [a]$ gives an isomorphism
$H^p(M^*_p)\cong \hat{\mathbf{H}}^{p-1}_{smooth}(X)$.

\end{proof}

It is shown in \cite{HLZ} \cite{HL1} that there is a natural isomorphism
$\hat{\mathbf{H}}^{p-1}(X)\cong\hat{H}^{p-1}(X)$, so
we get \cite[Proposition 1.5.7.]{Br} as a corollary.

\begin{cor}
$$H^p_{\mathcal{D}}(X,\mathbb{Z}(p)^{\infty})\cong\hat{H}^{p-1}(X).$$
\end{cor}

In fact, $\bigoplus_{p}H^p_{\mathcal{D}}(X,\mathbb{Z}(p)^{\infty})\subset
\bigoplus_{p,q}H^q_{\mathcal{D}}(X,\mathbb{Z}(p)^{\infty})$ is a subring, where the product coincides
with the products on spark characters and differential character, i.e. we have the following ring
isomorphism:
\begin{thm}
$$H^*_{\mathcal{D}}(X,\mathbb{Z}(*)^{\infty})\cong\hat{\mathbf{H}}^{*}(X)\cong\hat{H}^{*}(X).$$
\end{thm}

\begin{proof}
It is shown in \cite{HLZ} that $\hat{\mathbf{H}}^*(X)$ and $\hat{H}^*(X)$ are isomorphic as rings.
So we only need to verify that the product on $H^*_{\mathcal{D}}(X,\mathbb{Z}(*)^{\infty})$ agrees
with the product on $\hat{\mathbf{H}}^{*}(X)$.

We can make use of the isomorphism:
$$H^p_{\mathcal{D}}(X,\mathbb{Z}(p)^{\infty})\cong H^p(M^*_p)\cong \hat{\mathbf{H}}^{p-1}_{smooth}(X).$$
First, fix $$\alpha\in H^p_{\mathcal{D}}(X,\mathbb{Z}(p)^{\infty})\text{ and } \beta\in H^q_{\mathcal{D}}(X,\mathbb{Z}(q)^{\infty}),$$ and let
$$\tilde{a}=r+a=r+\sum_{i=0}^{p-1}a^{i,p-1-i}\in M^p_p \text{ be a representative of }\alpha$$
and  $$\tilde{b}=s+b=s+\sum_{i=0}^{q-1}b^{i,q-1-i}\in M^q_q\text{ be a representative of }\beta$$
where
$$r\in C^p(\mathcal{U},\mathbb{Z}), a^{i,p-1-i}\in C^i(\mathcal{U},\mathcal{E}^{p-1-i}),$$
and
$$s\in C^q(\mathcal{U},\mathbb{Z}), b^{i,q-1-i}\in C^i(\mathcal{U},\mathcal{E}^{q-1-i}).$$

On one hand, we calculate $\alpha\cup\beta$ by original product formula (See Appendix):
$$\alpha\cup\beta=[r\cup\tilde{b}+a\cup db^{0,q-1}].$$

On the other hand, let $[a]$ and $[b]$ be the image of $\alpha$ and $\beta$
under the isomorphism $H^k(M^*_k)\cong \hat{\mathbf{H}}^{k-1}_{smooth}(X)$, $k=p,q$ with
spark equations $$Da=e-(-1)^pr \text{ and } Db=f-(-1)^qs \text{ where }
e=da^{0,p-1}, f=db^{0,q-1} \text{ are global forms}.$$
We apply product formula on $\hat{\mathbf{H}}^{*}_{smooth}(X)$, and get
$$[a][b]=[a\cup f+(-1)^p(-1)^pr\cup b]=[a\cup db^{0,q-1}+r\cup b]$$
which is the image of $[r\cup\tilde{b}+a\cup db^{0,q-1}]=
[a\cup db^{0,q-1}+r\cup b+r\cup s]$ under the isomorphism
of $H^{p+q}(M^*_{p+q})\cong \hat{\mathbf{H}}^{p+q-1}_{smooth}(X)$.

We get the products are the same.

\end{proof}

\section{Appendix. Products on Hypercohomology}
If we have three complexes of sheaves of abelian groups
$\mathcal{F}^*,\mathcal{G}^*,\mathcal{H}^*$
over a manifold $X$ and a cup product
$$\cup: \mathcal{F}^*\otimes\mathcal{G}^*\longrightarrow\mathcal{H}^*$$
which commutes with differentials, then $\cup$ induces
an product on their hypercohomology:
$$\cup: \mathbb{H}^*(X,\mathcal{F}^*)\otimes\mathbb{H}^*(X,\mathcal{G}^*)
\longrightarrow\mathbb{H}^*(X,\mathcal{H}^*).$$

Although the above fact is well known, it is hard to find reference on
how to realize the product on the cycle level. We write this appendix to give
an explicit formula of the induced product on \v{C}ech cycles which is useful
in the proof of Theorem 6.7.

Let us start from an easy case. Suppose we have three sheaves
$\mathcal{F},\mathcal{G},\mathcal{H}$ over $X$ and a cup product
$$\cup: \mathcal{F}\otimes\mathcal{G}\longrightarrow\mathcal{H}.$$

Fix an open covering $\mathcal{U}$ of $X$, we have \v{C}ech resolution of $\mathcal{F}$:
$$\mathcal{F}\rightarrow \mathcal{C}^0(\mathcal{U},\mathcal{F})\rightarrow
\mathcal{C}^1(\mathcal{U},\mathcal{F})\rightarrow\cdots$$
\v{C}ech cohomology of sheaf $\mathcal{F}$ with respect to $\mathcal{U}$ is defined as
$$\check{H}^*(\mathcal{U},\mathcal{F})
\equiv H^*(C^*(\mathcal{U},\mathcal{F}))$$
where $C^k(\mathcal{U},\mathcal{F})$ is group of global sections of sheaf
$\mathcal{C}^k(\mathcal{U},\mathcal{F})$.
When the open covering $\mathcal{U}$ is acyclic with respect to $\mathcal{F}$,
we have the canonical isomorphism
$$H^*(X,\mathcal{F})\cong\check{H}^*(\mathcal{U},\mathcal{F}).$$

Now we construct a morphism of complexes:
$$\phi:Tot(C^*(\mathcal{U},\mathcal{F})\otimes C^*(\mathcal{U},\mathcal{G}))
\longrightarrow C^*(\mathcal{U},\mathcal{F}\otimes \mathcal{G}).$$
For $a\in C^r(\mathcal{U},\mathcal{F})$, $b\in C^s(\mathcal{U},\mathcal{G})$
we put $$\phi(a\otimes b)_{i_0,\cdots,i_{r+s}}=a_{i_0,\cdots,i_r}\otimes b_{i_r,\cdots,i_{r+s}}$$

We fix the differential
$D=\delta_{\mathcal{F}}\otimes id+id\otimes(-1)^r\delta_{\mathcal{G}}$
on the total complex of double complex
$$\bigoplus_{r,s}C^r(\mathcal{U},\mathcal{F})\otimes C^s(\mathcal{U},\mathcal{G}),$$
where $\delta_{\mathcal{F}}$, $\delta_{\mathcal{G}}$ are \v{C}ech differentials on
$C^*(\mathcal{U},\mathcal{F})$ and $C^*(\mathcal{U},\mathcal{G})$ respectively.
It is easy to verify that $\phi$ is a chain map, i.e. commutative with differentials. Therefore,
$\phi$ induce a map $$\phi_*: H^*(Tot(C^*(\mathcal{U},\mathcal{F})\otimes C^*(\mathcal{U},\mathcal{G})))
\longrightarrow H^*(C^*(\mathcal{U},\mathcal{F}\otimes \mathcal{G}))\equiv
\check{H}^*(\mathcal{U},\mathcal{F}\otimes \mathcal{G}).$$

Also, $\cup: \mathcal{F}\otimes\mathcal{G}\longrightarrow\mathcal{H}$ induces a map
on \v{C}ech cohomology
$$\cup_*: \check{H}^*(\mathcal{U},\mathcal{F}\otimes \mathcal{G})\longrightarrow
\check{H}^*(\mathcal{U},\mathcal{H}).$$
Moreover, there is a natural map
$$\check{H}^*(\mathcal{U},\mathcal{F})\otimes \check{H}^*(\mathcal{U},\mathcal{G})
\longrightarrow H^*(Tot(C^*(\mathcal{U},\mathcal{F})\otimes C^*(\mathcal{U},\mathcal{G})))$$
induced by
$$C^*(\mathcal{U},\mathcal{F})\otimes C^*(\mathcal{U},\mathcal{G})
\longrightarrow Tot(C^*(\mathcal{U},\mathcal{F})\otimes C^*(\mathcal{U},\mathcal{G}))).$$

Finally, we get a map
$$\check{H}^*(\mathcal{U},\mathcal{F})\otimes \check{H}^*(\mathcal{U},\mathcal{G})
\longrightarrow H^*(Tot(C^*(\mathcal{U},\mathcal{F})\otimes C^*(\mathcal{U},\mathcal{G})))
\stackrel{\phi_*}\longrightarrow
\check{H}^*(\mathcal{U},\mathcal{F}\otimes \mathcal{G})\stackrel{\cup_*}\longrightarrow
\check{H}^*(\mathcal{U},\mathcal{H}).$$
And it is easy to see, for two \v{C}ech cycles
$a\in C^r(\mathcal{U},\mathcal{F})$ and $b\in C^s(\mathcal{U},\mathcal{G})$,
the cup product of $[a]$ and $[b]$ can be represented by $a\cup b$ which is defined by
$$(a\cup b)_{i_0,\cdots,i_{r+s}}=a_{i_0,\cdots,i_r}\cup b_{i_r,\cdots,i_{r+s}}$$

When the covering $\mathcal{U}$ is acyclic with respect to $\mathcal{F}$,
$\mathcal{G}$ and $\mathcal{H}$, we define
$$\cup: H^*(X,\mathcal{F})\otimes H^*(X,\mathcal{G})
\longrightarrow H^*(X,\mathcal{H})$$ by above construction.

Now we consider the case of hypercohomology of complexes of sheaves.
From now on, we assume the covering
$\mathcal{U}$ is acyclic with respect to all sheaves we deal with and identify \v{C}ech
(hyper)cohomology with sheaf (hyper)cohomology.

We shall prove the following theorem with explicit product formula on \v{C}ech cycles.
\begin{thm}
Let $(\mathcal{F}^*,d_{\mathcal{F}}),(\mathcal{G}^*,d_{\mathcal{G}}),(\mathcal{H}^*,d_{\mathcal{H}})$
be complexes of sheaves of abelian groups
over a manifold $X$. If there is a cup product
$$\cup: \mathcal{F}^*\otimes\mathcal{G}^*\longrightarrow\mathcal{H}^*$$
which commutes with differentials, then $\cup$ induces
an product on their hypercohomology:
$$\cup: \mathbb{H}^*(X,\mathcal{F}^*)\otimes\mathbb{H}^*(X,\mathcal{G}^*)
\longrightarrow\mathbb{H}^*(X,\mathcal{H}^*).$$

\end{thm}
\begin{proof}
Fix an open covering $\mathcal{U}$ of $X$, for complexes of sheaves
$\mathcal{A}^*$ ( $\mathcal{A}^*=\mathcal{F}^*$, $\mathcal{G}^*$ or
$\mathcal{H}^*$ ) we have \v{C}ech resolution of $\mathcal{A^*}$:
$$\mathcal{A^*}\rightarrow \mathcal{C}^*(\mathcal{U},\mathcal{A^*}).$$
Then
$$\mathbb{H}^*(X,\mathcal{A^*})\equiv
\mathbb{H}^*(Tot(\mathcal{C}^*(\mathcal{U},\mathcal{A^*})))
\equiv H^*(Tot(C^*(\mathcal{U},\mathcal{A^*})))$$
where $C^*(\mathcal{U},\mathcal{A^*})$ are groups of global sections of sheaves
$\mathcal{C}^*(\mathcal{U},\mathcal{A^*})$ and
$Tot(C^*(\mathcal{U},\mathcal{A^*}))$ is the total complex of double complex
$\bigoplus_{r,p}C^r(\mathcal{U},\mathcal{A}^p)$ with total differential
$D_{\mathcal{A}}=\delta+(-1)^rd_{\mathcal{A}}$.

Similar to the case when $\mathcal{F}^*$, $\mathcal{G}^*$ are
$\mathcal{H}^*$ are single sheaves, the cup product
$$\cup: \mathbb{H}^*(X,\mathcal{F}^*)\otimes\mathbb{H}^*(X,\mathcal{G}^*)
\longrightarrow\mathbb{H}^*(X,\mathcal{H}^*)$$
is defined as composition of three maps
$$\mathbb{H}^*(X,\mathcal{F}^*)\otimes \mathbb{H}^*(X,\mathcal{G}^*)
\longrightarrow H^*(Tot(C^*(\mathcal{U},\mathcal{F}^*)\otimes C^*(\mathcal{U},\mathcal{G}^*)))
\stackrel{\phi_*}\longrightarrow
\mathbb{H}^*(X,\mathcal{F}^*\otimes \mathcal{G}^*)\stackrel{\cup_*}\longrightarrow
\mathbb{H}^*(X,\mathcal{H}^*).$$

1) The first map is induced by
induced by
$$Tot(C^*(\mathcal{U},\mathcal{F}^*))\otimes Tot(C^*(\mathcal{U},\mathcal{G}^*))
\longrightarrow Tot(C^*(\mathcal{U},\mathcal{F}^*)\otimes C^*(\mathcal{U},\mathcal{G}^*)))$$
where $Tot(C^*(\mathcal{U},\mathcal{F}^*)\otimes C^*(\mathcal{U},\mathcal{G}^*))$
is the total complex of complex
$$\bigoplus_{r,s,p,q}(C^r(\mathcal{U},\mathcal{F}^p))\otimes (C^s(\mathcal{U},\mathcal{G}^q))$$
with differential $D=D_{\mathcal{F}}\otimes id+id\otimes (-1)^{r+p}D_{\mathcal{G}}$.

2) Now we construct $\phi$ which induces the second map $\phi_*$:
$$\phi:Tot(C^*(\mathcal{U},\mathcal{F}^*)\otimes C^*(\mathcal{U},\mathcal{G}^*))
\longrightarrow Tot(C^*(\mathcal{U},\mathcal{F}^*\otimes \mathcal{G}^*)).$$
For $a\in C^r(\mathcal{U},\mathcal{F}^p)$, $b\in C^s(\mathcal{U},\mathcal{G}^q)$
we define $\phi(a\otimes b)\in C^{r+s}(\mathcal{U},\mathcal{F}^p\otimes\mathcal{G}^q))$ by
$$\phi(a\otimes b)_{i_0,\cdots,i_{r+s}}=(-1)^{ps}a_{i_0,\cdots,i_r}\otimes b_{i_r,\cdots,i_{r+s}}.$$

Note that $\mathcal{F}^*\otimes \mathcal{G}^*$ is the total complex (of sheaves)
of double complex $\bigoplus_{p,q}\mathcal{F}^p\otimes \mathcal{G}^q$ with differential
$d_{\mathcal{F}\otimes\mathcal{G}} =d_{\mathcal{F}}\otimes id +id\otimes (-1)^pd_{\mathcal{G}}$.
And $Tot(C^*(\mathcal{U},\mathcal{F}^*\otimes \mathcal{G}^*))$ is the total complex of
$\bigoplus_{r,p,q} C^r(\mathcal{U},\mathcal{F}^p\otimes \mathcal{G}^q)$ with differential
$D_{\mathcal{F}\otimes\mathcal{G}}=\delta+(-1)^rd_{\mathcal{F}\otimes\mathcal{G}}$.

We have to verify that $\phi$ is a chain map, i.e. commutative with differentials.
In fact, the purpose that we put a sign $(-1)^{ps}$ in the definition of $\phi$ is to make
$\phi$ to be a chain map.

The calculation here is essentially same as the one in the proof of Proposition 4.2.
For an element
$a\otimes b\in C^r(\mathcal{U},\mathcal{F}^p)\otimes C^s(\mathcal{U},\mathcal{G}^q)$,

\begin{eqnarray*}
& &\phi(D(a\otimes b))\\
&=&\phi(D_{\mathcal{F}}a\otimes b+(-1)^{r+p}a\otimes D_{\mathcal{G}}b)\\
&=&\phi((\delta a+(-1)^rd_{\mathcal{F}}a)\otimes b+
(-1)^{r+p}a\otimes (\delta b+(-1)^sd_{\mathcal{G}}b))\\
&=&\phi(\delta a\otimes b+(-1)^rd_{\mathcal{F}}a\otimes b+
(-1)^{r+p}a\otimes\delta b+(-1)^{r+p+s}a\otimes d_{\mathcal{G}}b)\\
&=&(-1)^{ps}\delta a\otimes b+(-1)^{(p+1)s+r}d_{\mathcal{F}}a\otimes b+
(-1)^{p(s+1)+r+p}a\otimes\delta b+(-1)^{ps+r+p+s}a\otimes d_{\mathcal{G}}b\\
           &=&(-1)^{ps}\delta a\otimes b+(-1)^{ps+s+r}d_{\mathcal{F}}a\otimes b+
(-1)^{ps+r}a\otimes\delta b+(-1)^{ps+r+p+s}a\otimes d_{\mathcal{G}}b
               \end{eqnarray*}
Note that I abuse the natation $\otimes$ which have different meanings in
the domain and image of $\phi$.

\begin{eqnarray*}
& &D_{\mathcal{F}\otimes\mathcal{G}}(\phi(a\otimes b))\\
&=&D_{\mathcal{F}\otimes\mathcal{G}}((-1)^{ps}(a\otimes b))\\
&=&(\delta+(-1)^{r+s}d_{\mathcal{F}\otimes\mathcal{G}})((-1)^{ps}(a\otimes b))\\
&=&(-1)^{ps}\delta(a\otimes b)+(-1)^{ps+r+s}d_{\mathcal{F}\otimes\mathcal{G}}(a\otimes b)\\
&=&(-1)^{ps}(\delta a\otimes b+(-1)^ra\otimes\delta b)+
(-1)^{ps+r+s}(d_{\mathcal{F}}a\otimes b+(-1)^pa\otimes d_{\mathcal{G}}b)\\
&=&(-1)^{ps}\delta a\otimes b+(-1)^{ps+r}a\otimes\delta b+
(-1)^{ps+s+r}d_{\mathcal{F}}a\otimes b+(-1)^{ps+r+p+s}a\otimes d_{\mathcal{G}}b\\
&=&\phi(D(a\otimes b))
               \end{eqnarray*}

Therefore, $\phi$ is a chain map and
induces a map $$\phi_*: H^*(Tot(C^*(\mathcal{U},\mathcal{F}^*)\otimes C^*(\mathcal{U},\mathcal{G}^*)))
\longrightarrow H^*(C^*(\mathcal{U},\mathcal{F}^*\otimes \mathcal{G}^*))\equiv
\mathbb{H}^*(X,\mathcal{F}^*\otimes \mathcal{G}^*).$$

3) Because $\cup$ commutes with differentials, it is easy to see $\cup$ induces a chain map:
$$Tot(C^*(\mathcal{U},\mathcal{F}^*\otimes \mathcal{G}^*)) \longrightarrow
Tot(C^*(\mathcal{U},\mathcal{H}^*)).$$
Therefore, we have an induced map on hypercohomology
$$\cup_*: \mathbb{H}^*(X,\mathcal{F}^*\otimes \mathcal{G}^*) \longrightarrow
\mathbb{H}^*(X,\mathcal{H}^*).$$

From above process, we can realize the cup product on \v{C}ech cycles.
Assume $\alpha\in \mathbb{H}^k(X,\mathcal{F}^*)$ and
$\beta\in \mathbb{H}^l(X,\mathcal{G}^*)$. Let
$a=\sum_{r+p=k}a^{r,p}\in \bigoplus_{r+p=k}C^r(\mathcal{U},\mathcal{F}^p)$
be a \v{C}ech cycle representing $\alpha$ and
$b=\sum_{s+q=l}b^{s,q}\in \bigoplus_{s+q=l}C^s(\mathcal{U},\mathcal{G}^q)$
be a \v{C}ech cycle representing $\beta$.

We define $$a\cup b\equiv \sum_{r+p=k,s+q=l}(-1)^{ps}a^{r,p}\cup b^{s,q},$$
then $\alpha\cup\beta$ is represented by $a\cup b$.

\end{proof}

\begin{rmk}
The cup product
$$\cup: \mathcal{F}^*\otimes\mathcal{G}^*\longrightarrow\mathcal{H}^*,
\qquad \cup(a\otimes b)\equiv a\cup b$$
commutes with differentials, in other word,
$$d_{\mathcal{H}}(a\cup b)=\cup(d_{\mathcal{F}\otimes\mathcal{G}}(a\otimes b))
=\cup(d_{\mathcal{F}}a\otimes b+(-1)^{\deg a}a\otimes d_{\mathcal{G}} b)=
d_{\mathcal{F}}a\cup b+(-1)^{\deg a}a\cup d_{\mathcal{G}} b$$
i.e. $\cup$ satisfies the Leibniz rule.

In the proof of Theorem 6.1, we show the cup product on the
\v{C}ech cycle level
$$Tot(C^*(\mathcal{U},\mathcal{F}^*))\otimes Tot(C^*(\mathcal{U},\mathcal{G}^*))
\longrightarrow Tot(C^*(\mathcal{U},\mathcal{H}^*))$$
satisfies the Leibniz rule.
Therefore, we have a well-defined product on hypercohomology.
\end{rmk}

Let us go back to the cup product of smooth Deligne cohomology.

The cup product
$$\cup: \mathbb{Z}_{\mathcal{D}}(p)^{\infty}\otimes \mathbb{Z}_{\mathcal{D}}(q)^{\infty}
\rightarrow \mathbb{Z}_{\mathcal{D}}(p+q)^{\infty}$$ is defined by
$$x\cup y=\left\{
    \begin{array}{ll}
      x\cdot y & \hbox{if } \deg x=0;\\
      x\wedge dy & \hbox{if } \deg x>0 \hbox{ and } \deg y=q;\\
      0 & \hbox{otherwise.}
    \end{array}
  \right.$$

Assume $$\alpha\in H^p_{\mathcal{D}}(X,\mathbb{Z}(p)^{\infty})\text{ and } \beta\in H^q_{\mathcal{D}}(X,\mathbb{Z}(q)^{\infty}),$$ and let
$$\tilde{a}=r+a=r+\sum_{i=0}^{p-1}a^{i,p-1-i}\in M^p_p \text{ be a representative of }\alpha$$
and  $$\tilde{b}=s+b=s+\sum_{i=0}^{q-1}b^{i,q-1-i}\in M^q_q\text{ be a representative of }\beta$$
where
$$r\in C^p(\mathcal{U},\mathbb{Z}), a^{i,p-1-i}\in C^i(\mathcal{U},\mathcal{E}^{p-1-i}),$$
and
$$s\in C^q(\mathcal{U},\mathbb{Z}), b^{i,q-1-i}\in C^i(\mathcal{U},\mathcal{E}^{q-1-i}).$$

Then we calculate
\begin{eqnarray*}
\alpha\cup\beta &=&[\tilde{a}\cup \tilde{b}]\\
&=&[rs+\Sigma_i(-1)^{0\cdot i}r\cdot b^{i,q-1-i}+\Sigma_j(-1)^{(p-j)0}a^{j,p-1-j}\wedge db^{0,q-1}]\\
&=&[r\cdot \tilde{b}+a \wedge db^{0,q-1}]\\
&=&[r\cup\tilde{b}+a\cup db^{0,q-1}]
\end{eqnarray*}

Note that in the last line above, we use $\cup$ in the sense of Proposition 4.2.

\bigskip

\noindent {\sc Ning Hao }\\
Mathematics Department, SUNY at Stony Brook\\
Stony Brook, NY 11794, US\\
Email: {\tt nhao@math.sunysb.edu}

\end{document}